\documentclass{amsart}
\usepackage[shortcuts]{extdash}
\usepackage{hyperref}

\usepackage{cite}
\usepackage{csquotes}

\usepackage{color}
\usepackage{amssymb,amsmath,amsfonts,amsthm, latexsym,amscd}
\usepackage[all]{xy}
\usepackage{epic,eepic,float}

\usepackage[mathscr]{eucal}

\usepackage{rotating,indentfirst,array,varioref}
\usepackage{appendix,marginnote,tikz,pgf,mathtools}
\usepackage{setspace}
\usepackage{stmaryrd}

\usepackage[left=2cm,right=2cm,  top=2cm,bottom=2cm,bindingoffset=0cm]{geometry}

\usepackage{tikz-cd,pgf}

\def\res{\mathop{\mathrm{Res}}\limits}

\def\be{\begin{equation}}
\def\ee{\end{equation}}
\def\barr{\begin{array}}
\def\earr{\end{array}}

\newtheorem{dummy}{dummy}[section]

\newtheorem{Proposition}[dummy]{Proposition}
\newtheorem{Theorem}[dummy]{Theorem}
\newtheorem{remark}[dummy]{Remark}
\newtheorem{definition}[dummy]{Definition}
\newtheorem{example}[dummy]{Example}

\newcommand{\ba}{\begin{equation}\begin{aligned}}
\newcommand{\ea}{\end{aligned}\end{equation}}

\newcommand{\bml}{\begin{multline}}
\newcommand{\eml}{\end{multline}}

\newcommand{\bC}{\mathbb{C}}
\newcommand{\bD}{\mathbb{D}}

\newcommand{\bP}{\mathbb{P}}
\newcommand{\bQ}{\mathbb{Q}}
\newcommand{\bR}{\mathbb{R}}
\newcommand{\bZ}{\mathbb{Z}}

\newcommand{\one}{\mathbf{1}}

\newcommand{\cA}{\mathcal{A}}
\newcommand{\cB}{\mathcal{B}}
\newcommand{\cC}{\mathcal{C}}

\newcommand{\cE}{\mathcal{E}}
\newcommand{\cF}{\mathcal{F}}
\newcommand{\cH}{\mathcal{H}}
\newcommand{\cI}{\mathcal{I}}
\newcommand{\cL}{\mathcal{L}}

\newcommand{\cO}{\mathcal{O}}

\newcommand{\cU}{\mathcal{U}}
\newcommand{\cV}{\mathcal{V}}
\newcommand{\cX}{\mathcal{X}}

\newcommand{\dd}{\mathrm{d}}

\newcommand{\orb}{\mathrm{orb}}
\newcommand{\Spec}{\mathrm{Spec}}

\newcommand{\Def}{\mathrm{Def}}
\newcommand{\Ob}{\mathrm{Ob}}

\newcommand{\GL}{\mathrm{GL}}

\newcommand{\ch}{\mathrm{ch}}

\begin{document}

\title{Wall-crossing for K-theoretic quasimap invariants I}
\author{Konstantin Aleshkin}
\address{Konstantin Aleshkin, Department of Mathematics, Columbia University, 2990 Broadway, New York, NY 10027}
\email{aleshkin@math.columbia.edu}

\author{Chiu-Chu Melissa Liu} 
\address{Chiu-Chu Melissa Liu, Department of Mathematics, Columbia University, 2990 Broadway, New York, NY 10027}
\email{ccliu@math.columbia.edu}

\maketitle

\begin{abstract}
  We study K-theoretic GLSM invariants with one-dimensional gauge group and introduce elliptic central charges that depend
  on an elliptic cohomology class called an elliptic brane and a choice of level structure. These central charges have an integral
  representation related to an interpolation problem for the elliptic brane and satisfy natural q-difference equations. Integral representaions
  lead to the wall crossing statement for the central charges in two different phases of the GLSM. We explain how the orbifold structure
  leads to differences between wall crossing of the usual (K-theoretic) and elliptic central charges.
\end{abstract}

\flushbottom

\tableofcontents


\section{Introduction}

Let $V$ be a complex vector space, $G$ be a reductive group acting on $V$ and $\omega$ be a $G$-character.
Geometric invariant theory produces a semiprojective Deligne-Mumford (DM) stack $\cX := [V \sslash_{\omega} G]$. Quasimap invariants~\cite{CKM} are Gromov-Witten type
enumerative invariants associated to GIT quotients. These invariants exploit the fact that geometry of $\cX$ can be captured by the much simpler
space $V$ together with the group action on it: quasimap is a particular type of map to the Artin stack $[V/G]$ which generically maps to the
$\omega$-stable locus. It is known that under a certain choice of stability condition ($\epsilon = 0^{+}$ stability) genus zero quasimap
graph spaces are particulary simple - the domain curve must be irreducible. Corresponding invariants can often be computed explicitly.

There exists a generalization of quasimaps allowing to capture enumerate geometry of a critical locus of a function on $\cX$. This generalization
is called Gauge Linear Sigma Model (GLSM) following the paper of Witten~\cite{Witten}. Mathematical theory have been developed recently
by several groups of people:~\cite{FJR},~\cite{5aut},~\cite{FK}. Data of GLSM consists of the GIT quotient data together with
a holomorphic $G$-invariant function $W \; : \; V \to \bC$ called superpotential that is weighted homogeneous with respect to
additional $\bC^{*}_{R}$-action on $V$ that commutes with the $G$-action. GLSM depends on the stability condition $\omega$ in a locally-constant
way. Namely, the space of characters decomposes into chambers where the GIT quotient $\cX$ and the associated GLSM are different.
Such different models are sometimes called phases of the GLSM associated to different choices of $\omega$.
The simplest
example is so-called Landau-Ginzburg/Calabi-Yau correspondence~\cite{LGCY, CIR}. It is a manifestation of the fact that a Calabi-Yau hypersurface
in a (weighted) projective space and certain FJRW theory are different phases of the same GLSM. Relation between different phases of
GLSM is called GIT stability wall crossing. This wall crossing is directly related to
Crepant Transformation Conjecture~\cite{Ruan, Ruan2, Ruan3} for the underlying GIT quotients, e.g.~\cite{CIJ}.

GIT stability wall crossing for abelian GLSMs is studied in~\cite{AL1}.  
In~\cite{AL1} we compute genus-zero invariants of abelian GLSMs directly and show that they
have Mellin-Barnes type integral representation of the form proposed in~\cite{HoriRomo}. 
(In \cite{Shoemaker} genus-zero invariants of a large class of GLSMs are computed via relation to quasimap invariants of the underlying GIT quotients.)
Usually, when doing this type of wall crossing,
people work with cohomology-valued Givental $I$-functions or associated D-modules. Instead,
we work with central charges   because they are very convenient in the setting of wall crossing and mirror symmetry.\\

In this paper we study genus-zero K-theoretic quasimap GLSM invariants. GIT stability wall crossing for K-theoretic invariants is much less studied and to our knowledge
remains a mystery beyond some particular cases. Notably, there are results in holomorphic symplectic smooth case~\cite{AO, Nonabelian, Zhou} and
in the smooth balanced case~\cite{GW}. In particular, the classical Landau-Ginzburg/Calabi-Yau correspondence has not been worked out. In fact,
computations in~\cite{Wen} showed that it should break even for the quintic threefold. Here we consider the general GLSM with one-dimensional
gauge group $\cX = [V \sslash \bC^{*}]$. This includes the classical Landau-Ginzburg/Calabi-Yau correspondence and many features of the general
case, but is technically simpler. We note that our result also implies the wall crossing statement for the underlying GIT
if we put the superpotential and R-charges to zero. \\

We take the approach of~\cite{AL1}. First, we define elliptic central charges (Definition~\ref{def:ECC}) that are generalizations of the D-brane central
charges in quantum cohomology, see~\cite{Hosono}. A central charge can be thought of as a component of the $I$-function associated to a
B-type brane, that is a coherent sheaf for a GIT or a matrix factorization in the case of GLSMs. In both cases
the B-branes factor through the corresponding K-theory. Our elliptic central charges depend
on an elliptic cohomology class which provide a natural generalization of K-theory in the triality of generalized cohomology theories:
cohomology, K-theory and elliptic cohomology associated to the formal group laws on $\bC, \bC^{*}$ and elliptic curve $E = \bC^{*}/q^{\bZ}$.
In holomorphic symplectic setting the importance of elliptic cohomology was understood in~\cite{AO} in constructions of elliptic
stable envelopes. We also discuss appropriate modifications of the standard elliptic cohomology constructions in the orbifold case and
construct the orbifold elliptic Chern character, see Definition~\ref{def:orbEllipticChern}. The orbifold elliptic Chern character takes twisted
sectors into account similarly to how the usual orbifold Chern character~\cite{Kawasaki} does.

One important advantage of central charges is that they have natural integral representations of Mellin-Barnes and Euler type. We show that
our elliptic central charges also have integral representation called solid torus partition functions, 
K-theoretic $I$-functions and elliptic central charges have additional parameter called the level 
structure \cite{RZ, RWZ}. It turns out that appropriate choice of the level structure
is crucial for the existence of integral representations.

The solid torus should be thought of as a product of a disk and a circle. The circle appears due to the relation of K-theory of $\cX$ to cohomology of the loop space of $\cX$.
Solid torus partition functions have very simple representation and are morally elliptic central charges for counts of formal disks into the
critical locus of $W$ in the Artin stack $[V/G]$ (see also the discussion of vertex functions in~\cite{Nonabelian}, especially Section 3.2).
This relation is much more transparent in K theory than in cohomology.

Representation of central charges in terms of solid torus partition functions is closely related to the interpolation problem in elliptic
cohomology (see Theorem~\ref{th:EGRR}).
We follow the approach to elliptic cohomology~\cite{GKV} and conventions of~\cite{AO, Inductive, Nonabelian}.
Let $T$ be a torus acting on $V$, which induces a $T$-action on  the GIT quotient $\cX$. Then $T$-equivariant elliptic cohomology of a point is a product of elliptic curves $Ell_{T}(pt) = T/q^{\mathrm{cochar}(T)}$. $T$-equivariant elliptic cohomology of $\cX$ and $[V/G]$ are projective schemes over $Ell_{T}(pt)$. Stable locus inclusion
$\cX \subset [V/G]$ induces an inclusion map on the level of elliptic cohomology
$$ 
Ell_{T}(X) \subset Ell_T([V/G]) = Ell_{T \times G}(pt).
$$
An elliptic cohomology class on $\cX$ is a section of a line bundle over $Ell_{T}(X)$. Such a section
can be written in terms of theta functions. Solid torus partition function construction involves
interpolation of this section to a section of a line bundle $\cL$ over $Ell_{T \times G}(pt)$. If we enlarge the base $Ell_{T}(pt)$ by another
elliptic curve $E_{G}^{\vee}$, this interpolation exists and is unique under a degree
restriction on $\cL$ due to a basic fact about line bundles over abelian varieties. This additional degree of freedom corresponds to the K\"{a}hler
variable - variable that counts the degree in the K-theoretic $I$-function. The degree condition on the line bundle $\cL$ is called an elliptic
grade restriction rule since it is a natural generalization of the classical grade restriction rule. See, for example~\cite{HHP, HoriRomo, BFK, CIJS, AL1} for more information of the
classical grade restriction rule.

In Theorem~\ref{th:solidTorus} we show that elliptic central charge of an elliptic cohomology class has a solid torus partition function
representation that uses interpolation of that elliptic cohomology class. As a simple consequence we obtain the wall crossing
statement (Proposition~\ref{th:wallCrossing}). Wall crossing accounts to interpolating (elliptic cohomology class or corresponding central charge)
from one chamber and restricting it to another. Thus, analytic continuation of elliptic central charges is compatible with the elliptic
version of the grade restriction rule. We also show that elliptic central charges satisfy a certain
q-difference equation which we call quantum q-difference equation which is a 
q-analog of the Picard-Fuchs differential equation for the usual central charges.\\

We also discuss generalization of Landau-Ginzburg/Calabi-Yau correspondence. We study elliptic wall
crossing between degree $r$ hypersurface in an $n-1$-dimensional projective space (geometric phase) and 
a pair $([\bC^{n}/\mu_r], W)$ for a superpotential $W$ with a unique fixed point at the origin 
(Landau-Ginzburg phase).
When $\cO_{\bP^{n-1}}(-r)$ and $[\bC^n/\mu^r]$ are related by a crepant transformation when $n=r$.
This implies crepant wall crossing for the usual central charges. In particular, any central charge on
the one side can be represented as analytic continuation of a central charge on the other and vice versa.

For the elliptic central charges the situation changes: analogous property holds if $n=r^2$ for an
appropriate choice of the level structure whereas for $n=r$ the Landau-Ginzburg phase contains more
elliptic central charges than the geometric phase. This has the following geometric explanation.
Crepant transformation guarantees that K-theories of the respective phases have the same dimension
and one can construct an isomorphism using Fourier-Mukai transforms that reduce to the usual grade
restriction rule. K-theoretic $I$-functions and elliptic central charges are 
instead related to K-theory of inertia stack or elliptic cohomology correspondingly. 

\subsection{Notations}

\begin{enumerate}
	\item $\mu_r \simeq \bZ/r\bZ$ is a group of $r$-th roots of unity.
	\item $E[r] \simeq (\bZ/r\bZ)^2$ is a group of $r$-torsion points on the elliptic curve $E = \bC^*/q^{\bZ}$.
	\item Let $x \in X$. Then $[x] \in X/\sim$ denotes an equivalence class of $x$.
	\item For a given real number $x$ we denote its floor $\lfloor x \rfloor$ which is the largest integer not greater than $x$, and fractional part $\{x\} = x - \lfloor x \rfloor$.
	\item $pt \simeq \mathrm{Spec}(\bC)$ denotes a point.
	\item Let $T \simeq (\bC^*)^n$ be an algebraic torus with coordinates $t_1, \ldots, t_n$.
	Its character lattice is isomorphic to $\bZ^n$. Let $\alpha \in \bZ^n$. Then we use the notation $t^{\alpha}$ in the corresponding character, one-dimensional representation of $T$ defined by this character and corresponding element of $K_T(pt)$. In particular, $\bigoplus_{i} c_i t^{\alpha_i}, \;
	c_i \in \bZ$ denotes a $T$-representation and $\sum_i c_i t^{\alpha_i}$ - the corresponding K-theory
	class. Analytic functions in $t$ can be thought of as elements of completed K-theory by power series
	expansion.
	\item Let $G$ act on $X$ and $g \in G$. Then $X^g$ is a subset (scheme/stack) of $X$ fixed by $g$.
	\item If $H$ is a $\bZ$-module and $k$ is a field, then $H_k := H \otimes_{\bZ} k$. 
\end{enumerate}

\paragraph{\bf Acknowledgements}

We are grateful to Andrei Okounkov for useful remarks. The second named author is partially supported by  NSF DMS-1564497.

\section{Geometry of the model}

In this paper we consider Gauged Linear Sigma Models (GLSM) with the 1-dimensional gauge group $G$.
Recall that the GLSM data is given by the following 5-tuple: $(V, G, \bC^{*}_{R}, W, \omega)$.
Let $G \simeq \bC^{*}$. Then we can write $V = \bigoplus_{i=1}^{N} s^{D_{i}}$, where $s^{D_{i}}$ denotes a character of
$G$ and $D_{i} \in \bZ$. Without loss of generality we assume that

\begin{enumerate}
  \item All $D_{i} \ne 0$, 
  \item $D_{1}, \ldots, D_{N_{+}} >0,$ and $D_{N_{+}+1}, \ldots, D_{N} < 0$. We also denote $N_{-} = N-N_{+}$.
\end{enumerate}
In order to make the discussion cleaner we also assume that $0 < N_{+} < N$, so that the model has two nontrivial phases $\pm\omega > 0$.
We also introduce the notation $V = V_{+} \oplus V_{-}$, where the summands consist of positive (respectively negative) characters:
$V_{+} = \bigoplus_{i=1}^{N_{+}}s^{D_{i}}, \; V_{-} = \bigoplus_{i = N_{+}+1}^{N} s^{D_{i}}$.

The ambient spaces of both phases are total spaces of bundles over weighted projective spaces.
\begin{equation}
  \cX_{+} := [V \sslash_{1} G] = \bigoplus_{i > N_{+}} \cO(D_{i}) \to \bP[D_1: \ldots : D_{N_{+}}],
\end{equation}
where $\bP[D_1: \ldots : D_{N_{+}}]$ is a weighted projective space.
\begin{equation}
  \cX_{-} := [V \sslash_{-1} G] = \bigoplus_{i \le N_{+}} \cO(-D_{i}) \to \bP[-D_{N_{+}+1}, \ldots, -D_{N}],
\end{equation}

The classical example is when $D_{1} = D_{2} = \cdots = D_{n} = 1$ and $D_{N} = -r$. In this case we have
$\cX_{+} = \cO_{\bP^{n-1}}(-r)$ and $\cX_{-} = [\bC^{n}/\mu_{r}]$.
In the plus phase the GLSM describes quasimaps to a degree $r$ hypersurface in $\bP^{n-1}$ and in the minus phase FJRW theory of
$([\bC^{n}/\mu^{r}], W_{-})$ where $W_{-}$ is the restriction of superpotential.
In particular, if $n = r$, then the resulting GLSMs are Calabi-Yau. Below we consider the case $\omega > 0$ and
denote $\cX = \cX_{+}$ for concreteness since the other case $\omega < 0$ is exactly the same.\\

Let also $T$ be a maximal torus in $\mathrm{GL}(V)$ with coordinates $a_1, \ldots, a_N$
such that $V = \bigoplus a_i s^{D_i}$. Below we consider a $|\prod_i D_i|$-cover $T'$ of $T$
so that $a_i^{1/D_i}$ are coordinates on $T'$. Roots of equivariant variables $a_i$ should
be understood in this sense. 

\subsection{Inertia stack}

$\cX_{+}$ has $N_{+}$ torus fixed points $pt_{i}, \; i \le N_{+}$. Each of the points is a stacky point if $D_{i}>1$:
\begin{equation}
  pt_{i} = [(0\times \ldots \times 0 \times \underset{i}{\bC^{*}} \times 0\times \ldots \times 0)/\bC^{*}] \simeq [pt/\mu_{D_{i}}].
\end{equation}
Let $\Sigma$ be the toric fan of $\cX_{+}$. Set of cones is in one-to-one correspondence with the set of anticones $\cA_{+}$.
The correspodence is $\sigma = \sum_{i \in I'} \bR_{\ge 0}v_{i} \to I = [1..N] \backslash I'$.
Each cone corresponds to a closed substack $\cX_{+}(\sigma) \subset \cX_{+}$ defined by
\begin{equation}
  \cX_{+}(\sigma) = [(\{0\}^{I'} \times \bC^{I})^{ss}/G].
\end{equation}

Generic stabilizer $G_{\sigma}$ of $\cX_{+}(\sigma)$ is
\begin{equation}
  G_{\sigma} = \{g \in G \; | \; \{0\}^{I'} \times \bC^{I} \text{ is fixed under }g\}.
\end{equation}
If $\tau \subset \sigma$, then $\cX_{+}(\sigma) \subset \cX_{+}(\tau)$ and therefore $G_{\tau} \subset G_{\sigma}$.
The set $G_{\sigma}$ is naturally bijective to the set $\mathrm{Box}_{\sigma}$ \cite[Section 4]{BCS}, so we will use $\mathrm{Box}_{\sigma}$
and $G_{\sigma}$ interchangably. Define also $\mathrm{Box} = \bigcup_{\sigma \in \Sigma} \mathrm{Box}_{\sigma}$. This is a set of group elements
in $G$ that have nontrivial stabilizers on $V$ and thus produce stacky subvarieties in $\cX_{+}$. For a $v \in \mathrm{Box}$
denote by $g(v) \in G$ the corresponding group element. In our case $\mathrm{Box} = \bigcup_{i \le I_{+}} \mu_{D_{i}} \subset \bC^{*} \simeq G$.
\\

The inertia stack of a toric orbifold can be described using the box set:
\begin{equation}
  \cI\cX = \bigsqcup_{v \in \mathrm{Box}} \cX_{+,v}, \;\;\; \cX_{+,v} \simeq [(V^{ss})^{g(v)}/G].
\end{equation}
We denote the natural embedding $\cX_{v} \to \cX$ by $\iota_{v}$. \\

Given a line bundle $\cL \to \cX_{+}$ and $v \in \mathrm{Box}$ we have $g(v) \cdot \cL = e^{2\pi\sqrt{-1}c} \cL$ for a unique
  $0 \le c<1$. We define $\mathrm{age}_{v}(\cL) := c$ and $\mathrm{age}_{v}(\cL_{1} \oplus \cL_{2}) = \mathrm{age}_{v}(\cL_{1}) +
 \mathrm{age}_{v}(\cL_{2})$.

\subsection{Cohomology, K-theory and Elliptic cohomology}

All 3 cohomology theories satisfy Kirwan surjectivity. Let $u_{i} := c_{1}(a_{i}) \in H^{2}_{T}(pt)$ and $p = c_{1}(s)
\in H^{2}_{G}(pt)$. Then the cohomology and K-theory rings of $\cX_{+}$ with complex coefficients are given by
\begin{equation}
  H^{*}_{T}(\cX_{+})_{\bC} = \frac{\bC[u, p]}{\langle\prod_{i \le N_{+}} (D_{i}p-u_{i}) \rangle},
\end{equation}
\begin{equation}
  K_{T}(\cX_{+})_{\bC} = \frac{\bC[a^{\pm}, s^{\pm}]}{\langle \prod_{i\le N_{+}}(1-a_{i}^{-1}s^{-D_{i}}) \rangle}.
\end{equation}
We look at both cohomology and K-theory as structure sheaves over their (affine) spectra. Corresponding spectra are
cut out by equations $\prod_{i \le N_{+}} (D_{i}p-u_{i}) = 0$ and $\prod_{i\le N_{+}}(1-a_{i}^{-1}s^{-D_{i}})$ in
$\mathrm{Spec}(H^{*}_{T\times G}(pt)_{\bC}) \simeq \mathrm{Lie}(T) \oplus \mathrm{Lie}(G) \simeq \bC^{N+1}$ and
$\Spec K_{T\times G}(pt)_{\bC} \simeq T\times G \simeq (\bC^{*})^{N+1}$ respectively. All these schemes are defined over the spectrum of
$T$-equivariant cohomology/K-theory of a point.\\

Elliptic cohomology~\cite{GKV, AO} is a projective scheme, and elliptic cohomology classes are sections of sheaves over this scheme.
In particular, we define $E \simeq \bC^{*}/q^{\bZ}$ and for a torus $T$ we define $E_{T} \simeq T/q^{\mathrm{cochar}(T)} \simeq E^{\mathrm{rank}(T)}$.
Then
\begin{equation}
  Ell_{T \times G}(pt) \simeq E_{T} \times E_{G}.
\end{equation}

Equivariant elliptic cohomology of $\cX_{+}$ is a scheme that is
 cut out in $Ell_{T \times G}(pt)$ by the equation
\begin{equation} \label{eq:theta1}
  \prod_{i \le N_{+}} \theta(a_{i}^{-1}s^{-D_{i}}) = 0,
\end{equation}
where the theta function is defined as
\begin{equation}
  \theta(x) = \phi(x)\phi(q/x) = \prod_{i \ge 0}(1-q^{i}x)(1-q^{i+1}/x).
\end{equation}
The equation~\eqref{eq:theta1} geometrically is a zero locus of the section $\prod_{i \le N_{+}} \theta(a_{i}^{-1}s^{-D_{i}})$ of the line bunde
$\Theta(V_{+}^{\vee})\simeq \Theta(V_{+})$. This section is an (equivariant) elliptic Euler class of $V_{+}$. \\

Geometrically all 3 theories are transverse intersections of irreducible hypersurfaces, e.g. $1-a_{i}^{-1}s^{-D_{i}}, \; i \le N_{+}$ in K-theory.
This corresponds to the fixed point decomposition. Taking a finite cover of $T$ to allow roots of equivariant variables we get:
\begin{equation}
  (1-a^{-1}_{i}s^{D_{i}}) = \prod_{\zeta \in \mu_{D_{i}}}(1-\zeta a_{i}^{-1/D_{i}}s),
\end{equation}
where $K([pt/\mu_{D_{i}}]) \simeq \mu_{D_{i}}$.
\begin{equation}
  \theta(a^{-1}_{i}s^{D_{i}}) = \prod_{\zeta \in E[D_{i}]}\theta(\zeta a^{-1/D_{i}}s),
\end{equation}
where $Ell([pt/\mu_{D_{i}}]) \simeq E[D_{i}]$, and the notation $E[D_{i}]$ denotes the set of $D_{i}$-torsion points on $E$, that is solutions of
the equation $[x^{D_{i}}] = [1]$ on $E$; $E[D_i]$ is a finite subgroup of $E$ of order $D_i^2$. 

\subsection{Euler paring}

Let $\cF \in K_{T}(\cX_{+})$. Then we can compute $\chi(\cF)$ by equivariant localization.
\begin{equation}
  \chi(\cF) = \sum_{i \le N_{+}} \frac{\cF|_{pt_{i}}}{\Lambda^{*}T^{\vee}_{pt_{i}}} = \sum_{i \le N_{+}}\sum_{\zeta \in \mu_{D_{i}}} \frac{\cF}{\prod_{j \ne i \le N_{+}}(1-a_{j}^{-1}s^{-D_{j}})}\bigg|_{s = \zeta a_{i}^{-1/D_{i}}},
\end{equation}
where in the last formula we treat the expressions as rational functions in $s$ and evaluate them at specified values.
We can also rewrite this in the residue form:
\begin{equation} \label{eq:resEuler}
  \chi(\cF) = \sum_{k \le N_{+}}\sum_{\zeta \in \mu_{D_{k}}} \res_{s \to \zeta a_{k}^{-1/D_{k}}} \dd s \; \frac{\cF}{\Lambda^{*} V_{+}^{\vee}} =
  -(\res_{0}+\res_{\infty}) \; \dd s \; \frac{\cF}{\Lambda^{*} V_{+}^{\vee}},
\end{equation}
where we used that sum of residues of a rational differential form on a Riemann sphere is 0.\\

K-theoretic vertex function ($I$-function) takes values in the inertial K-theory $K_{T}(\cI\cX_{+})$ which is the correct object to consider once
once we have elliptic cohomology.\\

Let $\cF \in K_{T}(\cI\cX_{+})$ such that $\cF|_{\cX_v} = \cF_v$.  Then the Euler characteristic of $\cF$ is the sum of Euler characteristics over all
the components of inertia stack:
\begin{equation} \label{eq:orbEuler}
  \chi(\cF):= \sum_v\chi(\cX_v, \cF_v).
\end{equation}
Explicitly, we have
\begin{equation} \label{eq:orbEuler1}
  \chi(\cF) = \sum_{v \in \mathrm{Box}}\sum_{\substack{i \le N_{+}, \\ g(v)^{D_{i}} = 1}} \sum_{\zeta \in \mu_{D_i}} \frac{\cF_{v}}{\prod\limits_{\substack{j \ne i \le N_{+}, \\
    g(v)^{D_{j}}=1}} (1-a_{j}^{-1}s^{-D_{j}})}\bigg|_{s = \zeta a_{i}^{-1/D_{i}}}
\end{equation}

\section{Higgs branch: quasimap counting}

\subsection{LG quasimaps}

GLSM quasimaps are called Landau-Ginzburg quasimaps. Let $\Gamma \subset \GL(V)$ be a subgroup generated by $G$ and $\bC^{*}_{R}$. We view projection $\Gamma \to \bC^{*}_w$ as a character $\chi$ of $\Gamma$.
Such a map is captured by the commutative diagram:
\begin{equation}
  \begin{tikzcd}
    & \left[V /\Gamma\right] \arrow[d] \\
    \cC \arrow[ru, "f"] \arrow[r, "P"] \arrow["\omega^{log}",rd] & B\Gamma \arrow[d, "B\chi"] \\
    & B\bC^{*}_{w}
  \end{tikzcd}
\end{equation}
In the diagram above $\bC^*_w = \Gamma/G$.
Let us unwrap the data for $\cC \simeq \bP(a:1)$. The orbifold domain is required because the target space can have orbifold singularities, so we need at least one marked (possibly stacky) point.

\subsection{Geometry of the teardrop}
We represent
\begin{equation}
  \bP(a:1) = \{(x,y) \in \bC^{2} \backslash \{0\} \; | \; (x,y) \sim (\lambda^{a}x, \lambda y)\}.
\end{equation}
The orbifold point is $[\infty] := [1:0]$ and the other pole $[0]:=[0:1]$ is a usual point. This geometry is
sometimes called ``teardrop'' geometry. \\

We also consider a $\bC^{*}_{q}$ action on the teardrop:
\begin{equation}
  (x,y) \to (qx, y) \sim (x, yq^{-1/a}).
\end{equation}
Then $[0]$ and $[\infty]$ are the fixed points of this action. In what follows we will study
representation structure of bundles over the teardrop. We assign the weights $0$ and $-1/m$ to variables
$x,y$ in $\bC[x,y]$ (which is equivalent to a choice of linearization with respect to $\bC^{*}_{q}$ action).
The variable $q$ is called the loop variable.\\

Let $\cO_{\bP(a,1)}(1)$ denote the orbibundle with the total space

\begin{equation}
  \mathrm{tot}\, \cO_{\bP(a:1)}(1) = \{(x,y,z) \in (\bC^{2} \backslash \{0\}) \times \bC \; | \; (x,y,z) \sim (\lambda^{a}x, \lambda y, \lambda z)\}.
\end{equation}
We upgrade it to a $\bC^{*}_{q}$ equivariant line bundle as
\begin{equation}
  \cO_{\bP(a:1)}(n) \to \cO_{\bP(a:1)}((n \; \mathrm{mod} \; a)[\infty] + \lfloor n/a \rfloor [0]).
\end{equation}
We also use the following notations for the floor function and fractional part:
$\lfloor x\rfloor$ is the smallest integer smaller than $x$ and $\{x\} = x - \lfloor x\rfloor$.

We will be interested in cohomology and Euler characteristic of these sheaves.
Consider the maps
\begin{equation}
  \bP[a:1] \stackrel{\pi}{\to} \bP^{1} \to pt,
\end{equation}
where $\pi$ is the projection to the coarse moduli space.
We introduce the notation:
\begin{equation}
  \cO(n/a) = \cO_{\bP^{1}}(n/a) := \cO_{\bP[a:1]}((n \; \mathrm{mod} \; a)[\infty] + \lfloor n/a \rfloor [0])
\end{equation}
Sections of $\cO(n/a)$ are weighted polynomials of degree $n$ of the form
\begin{equation}
  \sum_{i=0}^{\lfloor n/a \rfloor} c_{i}x^{i}y^{n-ai},
\end{equation}
so
\begin{equation}
  H^{0}(\bP^{1}, \cO(n/a)) = \sum_{i=0}^{\lfloor n/a\rfloor} q^{-\{n/a\}-i} ,
\end{equation}
Using the Cech complex we can compute the $\bC^{*}_{q}$-equivariant Euler characteristic:
\begin{equation}
  \begin{aligned}
    \chi_{\bC^{*}_{q}}(\cO(n/a)) = &\frac{q^{\{-n/a\}}}{1-q^{-1}} + \frac{q^{-n/a}}{1-q} = \sum_{k \ge 0}q^{\{-n/a\}-k}-\sum_{k \ge 0}q^{-n/a-1-k} = \\
    = &\begin{dcases}
        \sum\limits_{k=0}^{\lfloor n/a\rfloor}q^{-\{n/a\}-k}, \;\;\; n \ge 0, \\
        0, \;\;\; n \in [-a, 0]\\
        \sum\limits_{k=0}^{-\lfloor n/a\rfloor-2}q^{k+1-\{n/a\}}, \;\;\; n < -a.
      \end{dcases}
  \end{aligned}
\end{equation}
Note that if $n \in [-a, -1]$, then $H^{0} = H^{1} = 0$.

Another way to compute the Euler characteristic is by the equivariant localization on the coarse moduli $\bP^{1}$:
\begin{equation}
    \chi_{\bC^{*}_{q}}(\bP^{1}, \cO(n/a)) = \frac{\cO(n/a)|_{[0]}}{\Lambda^{*}N_{[0]}^{\vee}}+\frac{\cO(n/a)|_{[\infty]}}{\Lambda^{*}N_{[\infty]}^{\vee}} 
    = \frac{q^{-\{n/a\}}}{1-q^{-1}} + \frac{q^{-n/a}}{1-q}.
\end{equation}

\subsection{Maps of formal disks}

We can formally consider domain as a $\bD := \mathrm{Spf}\bC[[y]]$ with $\bC^{*}_{q}$ action $y \to q^{-1/a}y$.
This can be thought of as a (formal) open chart
around $[1:0] = [\infty] \in \bP^{1}[a:1]$. We have
\begin{equation}
  \Gamma(\cO_{\bD}(n/a)) = \sum_{k \ge 0} c_{k}y^{n+ka},
\end{equation}
so
\begin{equation}
  \chi_{\bC^{*}_{q}}(\bD, \cO(n/a)) = \frac{q^{-n/a}}{1-q}.
\end{equation}

\subsection{Stacky loop space} 
Part of the data of a LG quasimap is a $G$-bundle $P$. Its degree is $P_{*}[\cC] \in H_{2}(B\Gamma), \; H_{2}(B\Gamma)_{\bQ} \simeq H_{2}(BG)_{\bQ} \oplus
H_{2}(B\bC^{*}_{R})_{\bQ} \simeq \mathrm{cochar}_{\bQ}(G) \oplus \bQ$. The quasimap condition fixes the second component to be just one,
so that we have $\beta_{\Gamma} = (\beta, 1)$. Let $P_{\beta} \to \bP[a:1]$ be the corresponding principal $\Gamma$-bundle. \\

Define quantum version of the space $V$:
\begin{equation}
  \cV_{\beta} := \bR\Gamma(P_{\beta} \times_{\Gamma} V), \;\;\; V_{\beta}:=H^{0}(\cV_{\beta}), \; W_{\beta}:=H^{1}(\cV_{\beta}).
\end{equation}
Morally, $V_{\beta}, W_{\beta}$ are deformation and obstruction spaces for LG quasimaps. In particular, the quantum loop space with obstruction bundle
is defined as
\begin{equation}
  \cX_{\beta}:=[V_{\beta} \sslash_{\omega} G], \;\;\; \mathrm{Ob}_{\beta}:=[V_{\beta}^{\omega-ss} \oplus W_{\beta}/G].
\end{equation}
We also define the subspace of the quasimaps non-singular at infinity $\cX_{\beta}^{\circ}$. Let
\begin{equation}
  V_{\beta}^{\circ}:=\{s \in V_{\beta} \; | \; s(1:0) \in V^{\omega-ss}\}.
\end{equation}
We note that the last condition is independent of the trivialization of $P_{\beta} \times_{\Gamma} V$ because the unstable locus is a cone in $V$.
Then $\cX_{\beta}^{\circ} := [V_{\beta}^{\circ} \sslash_{\omega} G]$. Advantage of using this space is the fact that there is an evaluation map at
infinity:
\begin{equation}
  ev_{\infty} \; : \; \cX^{\circ}_{\beta} \to \cX_{v(\beta)}.
\end{equation}

We note that due to the fact that we consider $\cC$ of genus 0 with one marked point and $\epsilon=0^{+}$ condition our quasimap moduli space is
very simple, in particular, it is an explicit smooth DM-stack and the obstruction bundle is a vector (orbi-)bundle.\\

Let us compute $V_{\beta}, W_{\beta}$ as $\bC^{*}_{q}$-representations. As a $T_{(a)}\times G_{s}\times \bC^{*}_{w}$-representation we have:
\begin{equation}
  V = \sum_{i=1}^{N} a_{i}s^{D_{i}} w^{q_i/2},
\end{equation}
For the quantum version we have compute
\begin{equation}
  \begin{aligned}
    &V_{\beta} =\sum_{i=1}^{N}H^{0}\left(\bP^{1}, \cO(D_{i} \beta-q_{i}/2)\right) , \\
    &W_{\beta} =\sum_{i=1}^{N}H^{1}\left(\bP^{1}, \cO(D_{i} \beta-q_{i}/2)\right) .
  \end{aligned}
\end{equation}
Let \
$$ d_{i}(\beta) := D_{i} \beta - q_{i}/2.$$

Using the formulas for the Euler characteristic from above from we compute:
\begin{equation} \label{eq:defObs}
  \begin{aligned}
    &V_{\beta} =\sum_{i=1}^{N}\left(a_{i}s^{D_{i}}\sum_{k=0}^{\lfloor d_{i}(\beta)\rfloor} q^{-i-\{d_{i}(\beta)\}} \right), \\
    &W_{\beta} =\sum_{i=1}^{N}\left(a_{i}s^{D_{i}}\sum_{k=0}^{-\lfloor d_{i}(\beta)\rfloor-2} q^{i+1-\{d_{i}(\beta)\}} \right).
  \end{aligned}
\end{equation}
Note, that $\{d_{i}(\beta)\} = \mathrm{age}_{v(\beta)}(U_{i}^{-1})$.

Define a vector subspace $Z_{\beta} \subset V$ by the equation
\begin{equation}
  Z_{\beta} = \bigcap_{i, \; d_{i}(\beta) \notin \bZ_{\ge 0}} \{x_{i} = 0\}.
\end{equation}
Value of a LG quasimap of degree $\beta$ must be in this subspace. We also define
\begin{equation}
  \cF_{\beta} := [Z_{\beta} \sslash_{\omega} G] \subset \cX_{\beta}.
\end{equation}
We use $\bC^{*}_{q}$-localization to compute the vertex function. $(\cX^{\circ}_{\beta})^{\bC^{*}_{q}}$ consists of constant maps. Moreover, we have
\begin{equation}
  ev_{\infty} \; : \; (\cX^{\circ}_{\beta})^{\bC^{*}_{q}} \simeq \cF_{\beta} \subset \cX_{v(\beta)},
\end{equation}
where the map $\beta \to v(\beta)$ is a natural map constructed as follows. $\beta \in H_{2}(BG)_{\bQ} \simeq \mathrm{cochar}(G)_{\bQ}$ defines a group element
$g(\beta) \in G$ by the formula $e^{2\pi\sqrt{-1}\beta}$. $\cX_{v(\beta)}$ is the inertia stack component fixed by $g(\beta)$.

\subsection{Virtual localization and  $I$-function}
$\bC^{*}_{q}$ acts fiberwise on bundles over $\cF_{\beta}$.
\begin{equation}
  \begin{aligned}
    &\mathrm{Def}_{\beta} := T_{\cX_{\beta}^{\circ}}|_{\cF_{\beta}} = \mathrm{Def}_{f}+\mathrm{Def}_{m}, \\
    &\mathrm{Ob}_{\beta}|_{\cF_{\beta}} = \mathrm{Ob}_{f}+\mathrm{Ob}_{m},
  \end{aligned}
\end{equation}
where the fixed part $f$ and the moving part $m$ denote the subspaces of weight zero and nonzero under the $\bC^{*}_{q}$ action.

We define the the vertex function or K-theoretic I-function as a pushforward of the localized virtual structure sheaf:
\begin{equation} \label{eq:virtLoc}
  I^{K} := \sum_{\beta \in \mathrm{Eff}} \frac{\theta(z^{-1})}{\theta(q^{-\{\beta\}}s^{-1}z^{-1})}z^{\lfloor\beta\rfloor}(ev^{\beta}_{\infty})_{*} \left( \frac{\cO_{\cF_{\beta}}}{\Lambda^{*} N^{\vee}_{vir}}\right) \cdot \cO_{\cX_{v(\beta)}},
\end{equation}
where $s \in G$ is a variable dual to $z$ and

\begin{enumerate}
  \item The summation range $\mathrm{Eff}$ is a set of effective curve classes. In our case it is
        $$\mathrm{Eff} = \bigcup_{i \le N_{+}} \frac{1}{D_{i}}\bZ_{\ge 0}, $$
        where the summation is over the fixed points in $\cX_{+}$.
  \item  The virtual structure sheaf of $\cF_{\beta}$ defined as
        \begin{equation}
          \cO_{\cF_{\beta}} \otimes \Lambda^{*} \mathrm{Ob}_{f}.
        \end{equation}
        Formula~\eqref{eq:defObs} to see that $W_{\beta}$ does not have a fixed part, so it is equal to the usual structure sheaf
        $\cO_{\cF_{\beta}}$.
  \item $N_{vir} = \Def_{m}-\Ob_{m}$ is a virtual normal bundle.

  \item $\cO_{\cX_{v(\beta)}}$ is the structure sheaf of the corresponding component of the inertia stack of $\cX$.
\end{enumerate}
\begin{remark}
  The significance of the prefactor in theta functions and floor function in the power will become clear later. For now we note the following
  relevant facts:
  \begin{enumerate}
    \item The function
          \begin{equation}
            \frac{\theta(z^{-1})}{\theta(q^{-\{\beta\}}s^{-1}z^{-1})}
          \end{equation}
          should be thought as an elliptic version of $z^{\{\beta\}}s^{-\ln(z)/\ln(q)}$ which is the usual prefactor in the I-functions.
    \item Let $T_{z}$ be an operator $f(z) \to f(qz)$. We have the following relation:
          \begin{equation}
            T_{z} \left( \frac{\theta(z^{-1})}{\theta(q^{-\{\beta\}}s^{-1}z^{-1})}z^{\lfloor\beta\rfloor} \right) =
            (q^{\beta}s) \cdot \frac{\theta(z^{-1})}{\theta(q^{-\{\beta\}}s^{-1}z^{-1})}z^{\lfloor\beta\rfloor}.
          \end{equation}
          This is the same transformation property as
          \begin{equation}
            T_{z}(z^{\beta}s^{\frac{-\ln(z)}{\ln(q)}}) = (q^{\beta}s) \cdot z^{\beta}s^{\frac{-\ln(z)}{\ln(q)}}.
          \end{equation}
  \end{enumerate}
\end{remark}

\begin{remark}
  In the definition of the I-function the exact form of the prefactor with theta functions is not important. Indeed, every meromorphic
  section $s$ of the same
  line bundle over $\cE_{G} \times \cE_{G}^{\vee}$ would be as good as $\theta(z^{-1})/\theta(q^{-\{\beta\}} s^{-1} z^{-1})$ for the results below
  (this section must be compatible
  with orbifold elliptic Chern character, see definition~\ref{def:orbEllipticChern}).
  This line bundle is degree zero over each component and is nontrivial. More precisely,
  it is defined by the transformation properties in the remark above.
  However, fractional powers and logarithms will ruin the integral representation
  below. In fact, one can even choose a different nontrivial degree 0 over each component line bundle, but then the q-difference equations satsified
  by the I-function will change as well.
\end{remark}

In order to compute the I-function we compute the virtual normal bundle first.
Below we will use the following formulas:
\begin{equation}
  \Lambda^{*}\left(\frac{q^{x}}{1-q}\right) = \Lambda^{*}(\sum_{k\ge 0} q^{x+k}) = \prod_{k\ge 0}(1-q^{x+k}) = \phi(q^{x}),
\end{equation}
and
\begin{equation}
  \Lambda^{*}\left(\frac{q^{x}}{1-q^{-1}}\right) = \Lambda^{*}(-\sum_{k\ge 0} q^{x+1}) = \prod_{k\ge 0}\frac{1}{(1-q^{x+1+k})} = \frac{1}{\phi(q^{x+1})}.
\end{equation}
Moreover,
\begin{equation}
  \begin{aligned}
    &\chi_{\bC^{*}_{q}}(\bP^{1}, \cO(x))  =
       \frac{\phi(q^{-x})}{\phi(q^{1-\{x\}})}, \\
    &\det \, (a\chi_{\bC^{*}_{q}}(\bP^{1}, \cO(x))) =
    \begin{dcases}
      & (aq^{-\{x\}})^{\lfloor x\rfloor+1}q^{-\frac{\lfloor x\rfloor(\lfloor x\rfloor+1)}{2}}, \;\;\; x > 0, \\
      & (aq^{1-\{x\}})^{\lfloor x\rfloor+1}q^{-\frac{(\lfloor x\rfloor+2)(\lfloor x\rfloor+1)}{2}}=
      (aq^{-\{x\}})^{\lfloor x\rfloor+1}q^{-\frac{\lfloor x\rfloor(\lfloor x\rfloor+1)}{2}}, \;\;\; x < -1,
    \end{dcases}
  \end{aligned}
\end{equation}
We can rewrite the last formula using theta functions:
\begin{equation}
  \det(a \chi_{\bC^{*}_{q}}(\bP^{1}, \cO(x))) =
    \frac{\theta(-aq^{-x})}{\theta(-aq^{1-\{x\}})} = (-1)^{\lfloor x\rfloor}\frac{\theta(aq^{-x})}{\theta(aq^{1-\{x\}})}.
\end{equation}

Recall that
\begin{equation}
  \iota_{v} \; : \; \cX_{v} \to \cX
\end{equation}
is the embedding of the inertia stack component $\cX_{v}$ into $\cX$.

We have:
\begin{equation}
  \Def-\Ob = \sum_{i} \iota^{*}_{v}a_{i}s^{D_{i}}\left(\frac{q^{-d_{i}(\beta)}}{1-q}+\frac{q^{-\{d_{i}(\beta)\}}}{1-q^{-1}}\right) =
  \sum_{i} \iota^{*}_{v}a_{i}s^{D_{i}}\left(\frac{q^{-d_{i}(\beta)}}{1-q}-\frac{q^{1-\{d_{i}(\beta)\}}}{1-q}\right), \;\;\; v = v(\beta).
\end{equation}
The fixed part is just the part that does not depend on $q$. For each $i$ such summands appear only when $d_{i}(\beta) \in \bZ_{\ge 0}$.
Note that $N_{\cF_{\beta}/\cX} = \sum_{i, \; d_{i}(\beta) \notin \bZ \ge 0}a_{i}s^{D_{i}}$ and
$$N_{\cF_{\beta}/\cX_{v(\beta)}} = \sum_{i, \; d_{i}(\beta) \in \bZ < 0}a_{i}s^{D_{i}}.$$
Then using that $\{-x\}=1-\{x\}$ if $x \notin \bZ$ and $\{-x\} = \{x\}=0$ if $x \in \bZ$ we can write the moving part as:
\begin{equation}
  (\Def-\Ob)_{m} = \sum_{i} \iota^{*}_{v}a_{i}s^{D_{i}}\left(\frac{q^{-d_{i}(\beta)}}{1-q}-\frac{q^{\{-d_{i}(\beta)\}}}{1-q}\right)+N_{\cF_{\beta}/\cX_{v(\beta)}}.
\end{equation}
The (virtual) bundles $N_{\cF_{\beta}/\cX_{v(\beta)}}$ and $(\Def-\Ob)_{m}$ are defined on $\cF_{\beta}$, but actually they are pullbacks of bundles
from $\cX$ via the embedding $\cF_{\beta} \subset \cX_{v(\beta)}$. Thus, we can use the formula $\iota_{*}\iota^{*} \alpha =
\alpha/\Lambda^{*}N_{\iota}$ to compute the pushforwards. Abusing notations by denoting the bundles on $\cX_{v(\beta)}$ by the same letter we compute:
\begin{equation}
  (ev^{\beta}_{\infty})_{*}\left( \frac{\cO_{\cF_{\beta}}}{\Lambda^{*}(\Def_{m}-\Ob_{m})^{\vee}} \right) =
  \frac{\Lambda^{*}N^{\vee}_{\cF_{\beta}/\cX_{v(\beta)}}}{\Lambda^{*}(\Def_{m}-\Ob_{m})^{\vee} \otimes \Lambda^{*}N^{\vee}_{\cF_{\beta}/\cX_{v(\beta)}}}
   = \prod_{i}\frac{\phi(\iota^{*}_{v}a^{-1}_{i}s^{-D_{i}}q^{\{-d_{i}(\beta)\}})}{\phi(\iota^{*}_{v}a^{-1}_{i}s^{-D_{i}}q^{-d_{i}(\beta)})}.
\end{equation}
Therefore, the vertex function is
\begin{equation}
  I^{K} = \sum_{\beta \in \mathrm{Eff}}\frac{\theta(z^{-1})}{\theta(q^{-\{\beta\}}s^{-1} z^{-1})} z^{\lfloor\beta\rfloor}\prod_{i}\frac{\phi(\iota^{*}_{v(\beta)}a^{-1}_{i}s^{-D_{i}}q^{\{-d_{i}(\beta)\}})}
  {\phi(\iota^{*}_{v(\beta)}a^{-1}_{i}s^{-D_{i}}q^{-d_{i}(\beta)})}\cdot \cO_{\cX_{v(\beta)}}.
\end{equation}

It is more convenient to work with the $\cH^{K}$-function:
\begin{equation} \label{eq:HK}
  \cH^{K} := I^{K} \cdot \hat{\Gamma}_{K} = \sum_{\beta \in \mathrm{Eff}}\frac{\theta(z^{-1})}{\theta(q^{-\{\beta\}}s^{-1}z^{-1})} z^{\beta}\frac{\prod_{i, \; \mathrm{age}_{v}(U^{i}) = 0}(1-\iota^{*}_{v}a^{-1}_{i}s^{-D_{i}})}{\prod_{i}\phi(\iota^{*}_{v(\beta)}a^{-1}_{i}s^{-D_{i}}q^{-d_{i}(\beta)})} \cdot
 \cO_{\cX_{v(\beta)}} ,
\end{equation}
where we introduced the K-theoretic Gamma-class
\begin{equation}
  \hat{\Gamma}_{K}:=\sum_{v \in \mathrm{Box}}\frac{\prod_{i, \; \mathrm{age}_{v}(U^{i}) = 0}(1-\iota^{*}_{v}a^{-1}_{i}s^{-D_{i}})}{\prod_{i \le N}\phi(\iota^{*}_{v}a_{i}^{-1}s^{-D_{i}}q^{-\mathrm{age}(U^{i})})} \cdot \cO_{\cX_{v}}.
\end{equation}
It is tempting to use the following interpretation of the $\cH^{K}$-function.
One gets the formula~\eqref{eq:HK} directly from the virtual localization formula~\eqref{eq:virtLoc} if we replace the Euler characteristic of
maps of spheres to the one of formal maps of disks:
\begin{equation}
  \chi_{\bC^{*}_{q}}(\bP^{1}, \cO_{\bD}(x)) \to \chi_{\bC^{*}_{q}}(\bD, \cO(x)), \;\;\; \frac{q^{-x}}{1-q}+\frac{q^{-\{x\}}}{1-q^{-1}} \to
  \frac{q^{-x}}{1-q}.
\end{equation}

\subsection{Level structure}
In quantum K-theory there is another way to modify the $I$-function which is called the level structure. Let $\cC_{\cX_{\beta}}
\stackrel{\pi_{\cC}}{\to} \cX_{\beta}$
be the universal curve and $f_{\cC} \; : \; \cC_{\cX_{\beta}} \to [V/\Gamma]$ be the universal section.
We can pullback elements of $K_{T}([V/\Gamma])$ to the universal curve and push them to the moduli space via projection.
Let $R = \sum_{\bar{l},n,m}R_{\bar{l},n,m}a^{\bar{l}}s^{n}w^{n} \in K_{T}([V/\Gamma])
\simeq K_{T\times\Gamma}(pt)$, where $a^{\bar{l}}, s^{n}$ and $w^{m}$ denote the characters of $T \times G \times \bC^{*}_{w}$ and $R_{\bar{l},m,n} \in \bZ$.
We define the level $R$ K-theoretic $I$-function to be
\begin{equation}
  I^{K}_{R} := \sum_{\beta \in \mathrm{Eff}} (-1)^{\sigma_{R}(\beta)}\frac{\theta(z^{-1})}{\theta(q^{-\{\beta\}}sz^{-1})}z^{\lfloor\beta\rfloor}(ev^{\beta}_{\infty})_{*} \left( \frac{\cO_{\cF_{\beta}}\otimes \pi_{\cC, *}(\det f_{\cC}^{*} \, R)}{\Lambda^{*} N^{\vee}_{vir}}\right) \cdot \cO_{\cX_{v(\beta)}},
\end{equation}
where $\sigma_{R}(\beta) := \sum_{\bar{l},m,n} R_{\bar{l},m,n}\lfloor m\beta-n\rfloor$.
In order to compute the I-function we first compute $\det f_{\cC}^{*} R$. We compute
\begin{equation}
  f_{\cC}^{*} R = \sum_{\bar{l},m,n} R_{\bar{l},m,n}a^{\bar{l}}\chi_{\bC^{*}_{q}}(\bP^{1}, \cO(m\beta-n)),
\end{equation}
and
\begin{multline}
    \det f_{\cC}^{*} R = \prod_{\bar{l},m,n}\left(\frac{\theta(-a^{\bar{l}}s^{m}q^{-m\beta+n})}{\theta(-a^{\bar{l}}s^{m}q^{1-\{m\beta-n\}})}\right)^{^{R_{\bar{l},m,n}}} =
    (-1)^{\sigma_{R}(\beta)} \prod_{\bar{l},m,n}\left(\frac{\theta(a^{\bar{l}}s^{m}q^{-m\beta+n})}{\theta(a^{\bar{l}}s^{m}q^{1-\{m\beta-n\}})}\right)^{^{R_{\bar{l},m,n}}} = \\
    = (-1)^{\sigma_{R}(\beta)}\prod_{\bar{l},m,n}
    \prod_{m\beta-n \in \bZ}(-a_{\bar{l}}s^{m})^{R_{\bar{l},m,n}}\times \left(\frac{\theta(a^{\bar{l}}s^{m}q^{-m\beta+n})}{\theta(a^{\bar{l}}s^{m}q^{\{-m\beta+n\}})}\right)^{^{R_{\bar{l},m,n}}},
\end{multline}
where we used the relation $1-\{x\} = \{-x\}$ if $x \notin \bZ$.
Notice that in the cohomological limit $q \to 1, \; a \to 1$ this becomes trivial. In K-theory this expression is elliptic in
nature.\\

We also define the level $R$ K-theoretic Gamma class:

\begin{equation}
  \hat{\Gamma}_{K, R}:=\sum_{v \in \mathrm{Box}}\prod_{\substack{\bar{l},m,n \\ g(v)^{n}e^{-2\pi\sqrt{-1} mq_{i}/2} =1}}(-\iota^{*}_{v}a_{\bar{l}}s^{n})^{-R_{\bar{l},m,n}}\frac{\prod_{i, \; \mathrm{age}_{v}(U^{i}) = 0}(1-\iota^{*}_{v}a_{i}^{-1}s^{-D_{i}})}{\prod_{i \le N}\phi(\iota^{*}_{v}a_{i}^{-1}s^{-D_{i}}q^{-\mathrm{age}(U^{i})})} \cdot \cO_{\cX_{v}}.
\end{equation}
In the most important case $R = V_{+}^{\vee} = \sum_{i \le N_{+}}a_{i}^{-1}s^{-D_{i}}$ the product over $\bar{l},n,m$ reduces to $\prod_{i, \; \mathrm{age}_{v}(U_{i})=0}(-\iota^{*}_{v}a_{i}s^{D_{i}})$
and
\begin{equation}
  \hat{\Gamma}_{K, V_{+}^{\vee}} = \sum_{v \in \mathrm{Box}}\frac{\prod_{i, \; \mathrm{age}_{v}(U^{i}) = 0}(1-\iota^{*}_{v}a_{i}s^{D_{i}})}{\prod_{i \le N}\phi(\iota^{*}_{v}a_{i}^{-1}s^{-D_{i}}q^{-\mathrm{age}(U^{i})})} \cdot \cO_{\cX_{v}}
\end{equation}

Correspondingly, the level $R$ H-function is
\begin{equation} \label{eq:hFunction}
  \cH^{K}_{R}:=I^{K}_{R} \cdot \hat{\Gamma}_{K, R}.
\end{equation}

\subsection{Central charges}

Similarly to the cohomological case, the natural object from the mirror symmetry point of view is not the I-function itself, but rather its pairing
with some B-model branes. In the cohomological case such an object was the central charge. Recall, that if $\cH$ is a cohomological $H$-function
\cite{Iritani}, then one can construct a central charge of $[\cB] \in K(\cX)$ or more generally of $\cB \in D^{b}(\cX)$:
\begin{equation}
  Z(\cB, R):=\langle \cH, \ch([\cB]^{\vee}) \rangle,
\end{equation}
where in the orbifold case we use orbifold intertia stack valued Chern character and orbifold cohomology pairing~\cite{CR}.

In the K-theory we need to upgrade the space of branes $K(\cX)$ and the pairing. It turns out, that it is most natural to
upgrade the K-theory $K(\cX)$ to some version of (equivariant) elliptic cohomology~\cite{Grojnowsky}.

\paragraph{\bf Elliptic branes}

 Let $\cL$ be a line bundle over
$E_T \times E_G$. Then its total space can be represented as
\begin{equation}
	\bC_{\chi(\cL)} \times T \times G/q^{\mathrm{cochar}(T \times G)},
\end{equation}
where $\chi(\cL)$ is a character of $G \times T$ that specifies the action on $\bC$. (Meromorphic)
sections of $\cL$ can be represented by quasiperiodic functions on $T \times G$ such that
for a given cocharacter $\sigma \in \mathrm{cochar}(T \times G)$:
\begin{equation} \label{eq:quasiPer}
	f(q^{\sigma}t) = \sigma(t)^{-1} q^{-\frac{|\sigma|^2}{2}}(-q^{-\frac12})^{\langle \sigma, \chi \rangle}f(t).
\end{equation}
Note that $\pi_1(E_T \times E_G)$ acts on the sections by rescaling according to the formula above.
We want to fix an identification of sections of $\cL$ with quasiperiodic functions.
We call an elliptic brane $\cE$ a section of a line bundle $\cL \to Ell_{T}(\cX)$ with such an 
identification.
Abusing notations below we assume that every elliptic cohomology class is
promoted to an elliptic brane, i.e., it corresponds to a fixed quasiperiodic functions.\\

As an example consider an elliptic curve $E$ with coordinate $x$ and the
line bundle $\cL = \cO([1]) \to E_{x}$. Then theta functions $\theta(q^{n}x)$ for different $n$
define the same section of this bundle. The transformation property of the theta function:
\begin{equation}
	\theta(q^nx) = (-x)^{-n}q^{-\frac{n(n-1)}{2}}\theta(x) = x^{-n}q^{-\frac{n^2}{2}}(-q^{-\frac12})^n
	\theta(x)
\end{equation}
is a particular case of the general transformation~\eqref{eq:quasiPer}.
A choice of an elliptic brane structure is a choice of $n \in \bZ$ in this case. \\

\paragraph{\bf Elliptic Chern character}

Recall, that $Ell_{T}(\cX)$ is a scheme over $Ell_{T}(pt) \simeq E_{T} \simeq T/q^{\mathrm{cochar}(T)}$.
Kirwan surjectivity implies
\begin{equation}
	Ell_{T}(\cX) = \{\prod_{i \le N_{+}} \theta(a_{i}^{-1}s^{-D_{i}}) = 0 \} \subset Ell_{T\times G}(pt).
\end{equation}
Elliptic cohomology elements on $\cX$ are sections of line bundles $\cL \to Ell_{T}(\cX)$. Sections of such a bundle can be represented by theta functions.

Consider the following diagram:
\begin{equation}
  \begin{tikzcd}
    \Spec(K_{T}(\cX)) \arrow[d] \arrow[r, "\ch_{K \to E}"] &Ell_{T}(\cX) \arrow[d] \\
    T \arrow[r, "\mathrm{mod} \; q^{\mathrm{cochar}(T)}"]& E_{T}
  \end{tikzcd}
\end{equation}
where
\begin{equation}
  \ch_{K \to E} \; : \; \Spec(K_{T}(pt)) \to Ell_{T}(pt)
\end{equation}
is the natural map making the diagram commute. This map is induced by the map
$$\Spec(K_{T}([V/G])) \simeq T \times G \;\;\; \stackrel{\text{mod }q}{\longrightarrow} \;\;\; E_{T} \times E_{G} \simeq Ell_{T}([V/G]).$$

Let $\cL \to Ell_{T}(\cX)$ be a line bundle. Sections of such line bundles are elliptic cohomology elements.
Then we can use pull-back via $\ch_{K \to E}$ to construct elliptic
version of Chern character map:
\begin{equation}
	(\ch_{K \to E})^{*} \; : \; \cL \to (\ch_{K \to E})^{*} \cL.
\end{equation}

 Moreover, $\Gamma(\hat{\cO}_{\mathrm{Spec}(K_T(\cX))})$ is isomorphic
to $\hat{K}_{T}(\cX)$, but the isomorphism is not canonical.
The notation $\hat{K}_{T}(\cX)$ denotes completion in the variable $q$. It might be necessary because
the map $\ch_{K \to E}$ is of infinite index. In other words, theta functions are transcendental functions.

A choice of elliptic brane structure on
sections of $\cL$ 
provides an embedding $\Gamma(\ch_{K \to E})^* \cL \subset \hat{K}_T(\cX)$ by sending a section
of $\cL$ to the quasiperiodic function on $\Spec(K_T(\cX)) \subset T \times G$ representing it.
Thus, we get
 \begin{equation}
 	(\ch_{K \to E})^{*} \; : \; \Gamma(\cL) \to \Gamma((\ch_{K \to E})^{*} \cL) \subset \hat{K}_{T}(\cX).
 \end{equation}


\paragraph{\bf Orbifold elliptic Chern character}

In the case of DM stacks it is more natural to define the image of the Chern character map to be the
inertial K-theory $\hat{K}(\cI\cX)$ instead of the usual one. This is parallel to how the usual Chern
character for DM stacks maps to the orbifold cohomology instead of the usual cohomology.
This is consistent with the fact that
\begin{equation}
  \Spec(H^{*}([pt/\mu_{r}])) = pt, \;\;\; \Spec(K([pt/\mu_{r}])) \simeq \mu_{r}, \;\;\;
  Ell([pt/\mu_{r}]) \simeq E[r].
\end{equation}
So if we want Chern character maps to be isomorphisms, we need to add twisted sectors to their image.\\

Recall, that we have $\cI \cX \simeq \bigsqcup_{v \in \mathrm{Box}} \cX_{v}$, and each $v$ corresponds to $g = g(v) \in G$ such that
$\cX_{v} = [V^g \sslash_\omega G]$. Let $\langle g \rangle \simeq \mu_{r}$ be a subgroup in
$G$ generated by $g$, where $r = \mathrm{ord}(g)$. Let us choose a logarithm of $g$ such that
$\log(g)/2\pi\sqrt{-1} \in \frac1r\bZ \cap [0,1)$.
We define the action of $\tilde{\mu}_{r} \simeq \frac1r\bZ$ on $G$ by $k \to q^{k\log(g)/2\pi\sqrt{-1}}$.
This reduces to the action of $\mu_{r}$ on $E_{G}$.

More abstractly we have
\begin{equation}
	Ell_{\langle g \rangle}(pt) \subset Ell_{T \times G}(pt).
\end{equation}
In particular, since $Ell_{T \times G}(pt)$ is an abelian variety and
$Ell_{\langle g \rangle}(pt)$ is its abelian subgroup, so $Ell_{\langle g \rangle}(pt)$ acts on
$Ell_{T \times G}(pt)$. This action preserves $Ell_T(\cX_v)$, where the latter is cut out in
$Ell_{T \times G}(pt)$ by the equation:
\begin{equation}
	\prod_{i \le N_+, \; \mathrm{age}_v(U_i) =0 }\theta(U_i) = 0.
\end{equation}
The equation is preserved by the action due to the condition $\mathrm{age}_v(U_i) =0$ and
 quasiperiodicity of theta functions.
This provides the action of $Ell_{\langle g \rangle}(pt) \simeq E[\mathrm{ord}(g)]$ on $Ell_{T}(\cX_v)$.
 
We have $E[\mathrm{ord}(g)] \simeq \mu_{\mathrm{ord}(g)}^{2}$ with generators $[g]$ and
$[q^{\ln(g)/2\pi\sqrt{-1}}]$ where the brackets denote an equivalence class modulo $q^{\bZ}$.
 As discussed above, we can upgrade it to the action of $\tilde{E}[\mathrm{ord}(g)] = \mu_r
  \times \frac{1}{r}\bZ$
on quasiperiodic functions on $\Spec(K_{T}(\cX)) \subset T \times G$. We denote this action by
$(l,k) \to g^{l}q^{k\log(g)/2\pi\sqrt{-1}}$.

Now we are ready to define the orbifold elliptic Chern character:
\begin{definition} \label{def:orbEllipticChern}
  Let $\cE$ be a $T$-equivariant elliptic brane on $\cX$. Then we define its elliptic Chern character to be
  \begin{equation}
    \ch^{E \to K}_{\mathrm{orb}}(\cE) := \sum_{v \in \mathrm{Box}} (q^{-\log(g(v))/2\pi\sqrt{-1}})^{*} (\ch_{K \to E})^{*}(\cE|_{\cX_{v}}) \cdot \cO_{\cX_{v}}.
  \end{equation}
\end{definition}
\begin{remark}
  The construction is parallel to the usual orbifold Chern character~\cite{Kawasaki}, but is formulated in a different
  language.

  Let $\cX$ be a smooth DM stack as above and $\cI\cX = \bigsqcup_{v \in \mathrm{Box}} \cX_{v}$ be its inertia stack.
  Let $\cL \to \cX$ be an orbifold line bunde. Then
  \begin{equation}
    \ch_{\mathrm{orb}}(\cL) = \sum_{v \in \mathrm{Box}} e^{2\pi\sqrt{-1} \mathrm{age}_{v}(\cL)} \, \ch(\cL|_{\cX_{v}}) \cdot \one^{H}_{v} \in
    H^{*}_{\mathrm{orb}}(\cX),
  \end{equation}
  where $0 \le \mathrm{age}_{v}(\cL) < 1$ is defined such that eigenvalue of $g(v)$ on $\cL|_{\cX_{v}}$ is
  $\exp(2\pi\sqrt{-1}\mathrm{age}_{v}(\cL))$ and
  $\one^{H}_{v}$ is one in the corresponding twisted sector.\\

  Let $\exp \; : \; \mathrm{Lie}(T) \to T$ be the exponential map. Consider the diagram
  \begin{equation}
    \begin{tikzcd}
      \Spec(H^{*}_{T}(\cX)) \arrow[d] \arrow[r, "\ch_{H \to K}"] &\Spec(K_{T}(\cX)) \arrow[d] \\
      \mathrm{Lie}(T) \arrow[r, "\exp"]& T
    \end{tikzcd}
  \end{equation}
  Then usual Chern character is
  \begin{equation}
    \ch(V) = (\ch_{H \to K})^{*} V,
  \end{equation}
  where we interpret $K(\cX)$ as $\Gamma(\cO_{\Spec(K(\cX))})$.
  Now, let $\cX_{v}$ be an inertia stack component. There is a natural $\langle g(v) \rangle$-action on
  $K(\cX_{v})$ since all sheaves on $\cX_{v}$ are representations of the generic stabilizer of a point on
  $\cX_{v}$.

  Then we can write the orbifold Chern character for
  $V$ as
  \begin{equation}
    \ch_{\orb}(V) = \sum_{v \in \mathrm{Box}}  (\ch_{H \to K})^{*}g^{-1}(v)^{*}(V|_{\cX_{v}}) \cdot \one_{v}^H.
  \end{equation}
\end{remark}
\begin{remark}
  We can also define the Chern character $\ch_{H \to E} = \ch_{K \to E} \circ \ch_{H \to K}$ and use it to pull back elliptic cohomology classes to
  $H^{*}(\cX)$. The orbifold version is constructed analogously and takes values in the double inertia
  stack $H^{*}(\cI\cI\cX)$.
\end{remark}

\begin{example}
  Let $\cX \simeq [\bC^{3}/\mu_{3}]$, where the action is diagonal. Then we have
  \begin{equation}
    K(\cX) \simeq \frac{\bC[x^{\pm}]}{\langle (1-x^{3}) \rangle}, \;\;\;
    Ell(\cX) \simeq \{\theta(x^{3}) = 0\} \subset E.
  \end{equation}
\end{example}
Let $\cE = \theta(x)$ be an elliptic brane on $\cX$. Then
\begin{equation}
  \ch_{K \to E}^{*}(\theta(x)) = \theta(x) = \theta(e^{2\pi\sqrt{-1}/3})\cdot \cO_{e^{2\pi\sqrt{-1}/3}}+ \theta(e^{4\pi\sqrt{-1}/3})\cdot \cO_{e^{4\pi\sqrt{-1}/3}},
\end{equation}
where in the right we view $\theta(x)$ as an element of $\bC[x^{\pm}]/\langle 1-x^{3} \rangle$ and in the
last equation we used the isomorphism $\bC[x^{\pm}]/\langle 1-x^{3} \rangle$ with the set of functions on
the solution of equation $1-x^{3}=0$. $\cO_{e^{2\pi\sqrt{-1}/3}}, \cO_{e^{4\pi\sqrt{-1}/3}}$ denote the
skyscrapper sheaves at corresponding points.

Elliptic Chern character is
\begin{equation}
  \ch^{E \to K}(\theta(x)) = \theta(x) \cdot \one_{0} + \theta(q^{1/3}x) \cdot \one_{1/3}+
  \theta(q^{2/3}x) \cdot \one_{2/3},
\end{equation}
where $\one_{v} = \cO_{[\bC^{3}/\mu^{3}]_{v}}$.

\paragraph{\bf Elliptic central charges}

\begin{definition} \label{def:ECC}
  We define the K-theoretic central charge of an elliptic brane $\cE \in Ell_{T}(\cX_{\beta})$ of
  level $R \in K_{T}([V/\Gamma])$:
  \begin{equation} \label{eq:kCharge}
    Z^{K}(\cE, R) := \chi(\cH^{K}_{R} \otimes \ch_{\mathrm{orb}}^{E \to K}(\cE)),
  \end{equation}
\end{definition}

\section{Solid torus partition function}

Let $E_{G}^{\vee}$ with coordinate $z$ be an elliptic curve dual to $E_{G}$. Recall that $z$ is the K\"{a}hler variable, that is $z$ counts
the degree of quasimaps. It naturally appears as an interpolation parameter for elliptic cohomology classes from $Ell_{T}(\cX_{+})$ to
$Ell_{T}([V/G]) = Ell_{T \times G}(pt)$. Addition of this parameter also guarantees that the interpolation exists under certain degree restrictions.

\begin{Theorem}[Elliptic grade restriction rule] \label{th:EGRR}
  Let $\cL \to Ell_{T \times G}(pt) \times E_{G}^{\vee}$ be a line bundle nontrivial over the second component $E_{G}^{\vee}$. Then the following interpolation problem has a unique solution for generic $z$
  \begin{equation} \label{eq:EGRR}
    0 \to \cL \otimes \Theta(-V_{+}) \to \cL \otimes \cO_{Ell_{T \times G}(pt)} \to i^{*}\cL \otimes \cO_{Ell_{T}(\cX_{+})} \to 0
  \end{equation}
  if $\mathrm{deg}_{G}(\cL) = \mathrm{deg}_{G}(V_{+})$.
\end{Theorem}
\begin{proof}
  Consider the long exact sequence corresponding to~\eqref{eq:EGRR}:
  \begin{equation}
    0 \to H^{0}(\cL \otimes \Theta(-V_{\pm})) \to H^{0}(\cL \otimes \cO_{Ell_{T \times G}(pt)}) \to H^{0}(i^{*}\cL \otimes \cO_{Ell_{T}(\cX_{\pm})}) \to H^{1}(\cL \otimes
    \Theta(-V_{\pm})) \to \cdots
  \end{equation}
  The line bundle $\cL \otimes \Theta(-V_{\pm})$ over elliptic curve $Ell_{G}(pt)$ has trivial $H^{0}$ and $H^{1}$. Line bundle over elliptic curve
  of degree 0 has nontrivial cohomology if and only if it is trivial. Since $\cL$ is nontrivial over $E^{\vee}_{G}$, the line bundle
  $\cL \otimes \Theta(-V_{+})$ is nontrivial for generic $z$.
\end{proof}

\begin{remark}
  Recall the usual grade restriction rule in our case (e.g.~\cite{AL1}).
  Consider projection
  \begin{equation}
    K_{T}([V/G]) \to K_{T}(\cX_{+}).
  \end{equation}
  This map is surjective (due to Kirwan surjectivity) but not injective. The other direction map can be thought as an interpolation of
  $\cB_{+} \in K_{T}(\cX_{+}) \simeq \Gamma(\cO_{\mathrm{Spec}(K_{T}(\cX_{+}))})$ to $K_{T \times G}(pt) \simeq \Gamma(\cO_{T \times G})$.
  The interpolation problem has a solution but it is not unique. A regular function on $T \times G$ is a linear combination of $T \times G$-characters.
  Let $K_{T\times G}(pt)|_{[L, L + D_{+}]}$ be a vector subspace spanned by $G$-characters in the range $[L, L+D_{+}-1] \subset \bZ \simeq \mathrm{char}(G)$.
  Grade restriction rule states that solution to the interpolation problem becomes unique on $K_{T \times G}(pt)|_{[L, L+D_{+}]}$, where
  $D_{+} = \sum_{i \le N_{+}} D_{i}$.
  \begin{equation}
    K_{T}(\cX_+) \simeq K_{T\times G}(pt)_{[L, L+D_{+}-1]}.
  \end{equation}
\end{remark}
Let us pick $a \in T$.
In figure~\ref{pic:interpolation} the infinite cylinder denotes $G \simeq \bC^{*} \simeq \mathrm{Spec}(K_{T \times G}(pt))|_{a}$. Then
$\mathrm{Spec}(K_{T \times G}(\cX_{+})|_{a}) = \bigcup_{i} p_{i} \subset \bC^{*}$. Interpolation problem is equivalent to finding
a regular function $f$ on $\bC^{*}$ that takes given values at $p_{i}$. We can compactify $\bC^{*}$ to $\bP^{1}$. The result is
called compactified K-theory~\cite{Inductive}. The problem of finding a function $f$ transforms to finding a line bundle $\cL \to \bP^{1}$
together with a local trivialization and its section $\tilde{f}$ such that the section takes required values at $p_{i}$.
Then given a local trivialization of $\cL$ the interpolation exists and is unique
for $f \in \Gamma(\bP^{1}, \cL)$ for $\cL = \cO(\sum_{i}p_{i}-\lambda [0]+\lambda[\infty])$, where $\lambda \in \mathrm{cochar}_{\bQ}(\bC^{*}) \simeq \bQ$
is a generic cocharacter. The twist by $\lambda$ does not change the degree of the line bundle but 
shifts its polytope (that is an interval in $\bR$) by $\lambda$ such that the number of integral points inside
coincides with its length.\\

The middle section of the cylinder is a fundamental domain of the elliptic curve $E = E_{G} \simeq Ell_{G}(pt)|_{a}$. Interpolation
problem is then to find a line bundle $\cL \to \cE$ with a section $s$ such that in a given trivialization $s(p_{i})$ are fixed numbers.
It is equivalent to finding a quasiperiodic function on $G$ with prescribed values at $p_{i}$.\\

\begin{figure}[h]
  \def\svgwidth{8cm}
\begingroup%
  \makeatletter%
  \providecommand\color[2][]{%
    \errmessage{(Inkscape) Color is used for the text in Inkscape, but the package 'color.sty' is not loaded}%
    \renewcommand\color[2][]{}%
  }%
  \providecommand\transparent[1]{%
    \errmessage{(Inkscape) Transparency is used (non-zero) for the text in Inkscape, but the package 'transparent.sty' is not loaded}%
    \renewcommand\transparent[1]{}%
  }%
  \providecommand\rotatebox[2]{#2}%
  \newcommand*\fsize{\dimexpr\f@size pt\relax}%
  \newcommand*\lineheight[1]{\fontsize{\fsize}{#1\fsize}\selectfont}%
  \ifx\svgwidth\undefined%
    \setlength{\unitlength}{482.4671083bp}%
    \ifx\svgscale\undefined%
      \relax%
    \else%
      \setlength{\unitlength}{\unitlength * \real{\svgscale}}%
    \fi%
  \else%
    \setlength{\unitlength}{\svgwidth}%
  \fi%
  \global\let\svgwidth\undefined%
  \global\let\svgscale\undefined%
  \makeatother%
  \begin{picture}(1,0.24690743)%
    \lineheight{1}%
    \setlength\tabcolsep{0pt}%
    \put(0,0){\includegraphics[width=\unitlength,page=1]{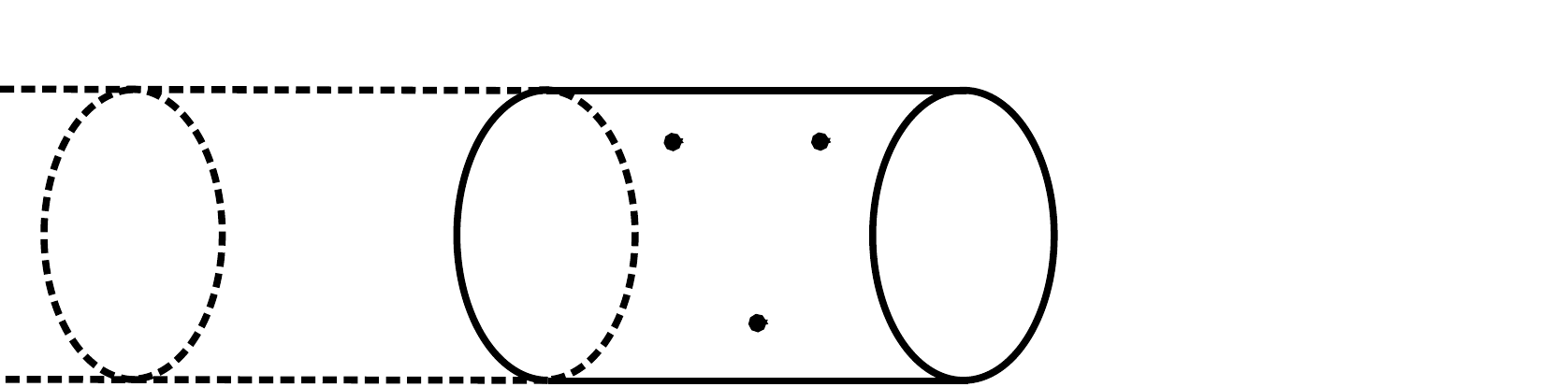}}%
    \put(0.35819962,0.20707763){\makebox(0,0)[lt]{\lineheight{1.25}\smash{\begin{tabular}[t]{l}$p_1$\end{tabular}}}}%
    \put(0.42087314,0.07113659){\makebox(0,0)[lt]{\lineheight{1.25}\smash{\begin{tabular}[t]{l}$p_2$\end{tabular}}}}%
    \put(0.53229497,0.20911338){\makebox(0,0)[lt]{\lineheight{1.25}\smash{\begin{tabular}[t]{l}$p_3$\end{tabular}}}}%
    \put(0,0){\includegraphics[width=\unitlength,page=2]{Interpolation.pdf}}%
  \end{picture}%
\endgroup%

  \caption{Interpolation}
  \label{pic:interpolation}
  \centering
\end{figure}

We can compute solution to the interpolation problem~\eqref{eq:EGRR} explicitly. Let $k \le N_{+}$ and $\zeta \in E[D_{k}]$.
Then $\cO_{Ell_{T}(\cX_{+})}$ is spanned over $\cO_{Ell_{T}(pt)}$ by the restriction of
\begin{equation}
  \tilde{\cE}_{z}^{(k, \zeta)} := \frac{\theta(\zeta a_{k}^{1/D_{k}}sz^{{-1}})}{\theta(z^{-1})\theta(\zeta a_{k}^{1/D_{k}}s)}\prod_{i \le N_{+}}\theta(a_{i}s^{D_{i}})
\end{equation}
to $\cX_{+}$.
$Ell_{T}(\cX_{+}) = \{\prod_{i \le N_{+}}\theta(a_{i}s^{D_{i}}) = 0\} \subset E_{T} \times E_{G}$ is a union of $\sum_{i \le N_{+}}D_{i}^{2}$ copies
of $E_{T}$. Then $\cE_{+}^{(k, \zeta)}|_{\cX_{+}}$ is nonzero only on one of these
copies cut out by the equation
\begin{equation}
  \theta(\zeta a_{k}^{1/D_{k}}s) = 0.
\end{equation}
We have
\begin{equation}
  \tilde{\cE}^{(k, \zeta)}_{z} \in \Gamma(E_{T} \times E_{G}, \; \cU_{(k, \zeta)} \otimes \Theta(V_{+})),
\end{equation}
where Poincare line bundle $\cU_{(k, \zeta)}$ is a bundle with the meromorphic section
\begin{equation}
  \frac{\theta(\zeta a_{k}^{1/D_{k}}sz^{{-1}})}{\theta(z^{-1})\theta(\zeta a_{k}^{1/D_{k}}s)}.
\end{equation}
These sections have a problem that they have different transformation properties with respect to $z$ for different $(k, \zeta)$. We upgrade them
to have the same transformation properties, that is to be sections of the same bundle over $E_{T} \times E_{G} \times E_{G}^{\vee}$.
Let
\begin{equation}
  \cE^{(k, \zeta)}_{z} := \frac{\theta(z^{-1})}{\theta(\zeta a_{k}^{1/D_{k}}z^{-1})}\tilde{\cE}^{(k, \zeta)}_{z} =
  \frac{\theta(\zeta a_{k}^{1/D_{k}}sz^{{-1}})}{\theta(\zeta a_{k}^{1/D_{k}}z^{-1})\theta(\zeta a_{k}^{1/D_{k}}s)}\prod_{i \le N_{+}}\theta(a_{i}s^{D_{i}}).
\end{equation}
Then
\begin{equation}
  \cE^{(k, \zeta)}_{z} \in \Gamma(E_{T} \times E_{G}, \; \cU \otimes \Theta(V_{+})),
\end{equation}
where Poincare line bundle $\cU \to E_{G} \times E^{\vee}_{G}$ has a meromorphic section
\begin{equation}
   \frac{\theta(sz^{-1})}{\theta(s)\theta(z^{-1})}.
\end{equation}

The price to pay for this simplification is that
\begin{equation}
  \cE^{(k, \zeta)}_{z}|_{\cX_{+}}
\end{equation}
is not a trivial bundle over the elliptic curve of Kahler variables $E^{\vee}_{G}$. More explicitly, the dependence on the Kahler variables is
the same as
\begin{equation}
  \frac{\theta(z^{-1})}{\theta(s^{-1}z^{-1})}.
\end{equation}

\begin{Theorem}[Central charge is equal to the solid torus partition function] \label{th:solidTorus}
  Let $\cX_{+}$ be as above and $\cE_{+} \in \Gamma(\cO_{Ell_{T}(\cX_{+})})$ and $\cL \simeq \cU \otimes \Theta(V_{+})$ satisfy the condition of
  Theorem~\ref{th:EGRR}. 
  Then if $Z^{K}(\cE_+, V_+^{\vee})$ converges then
  \begin{equation} \label{eq:solidTorus}
    Z^{K}(\cE_{+}, V_{+}^{\vee}) = \frac1{2\pi i}\oint_{\cC_{\delta}} \frac{\dd s}{s} \; \Gamma_{q} \cdot \cE_{z} = \mathrm{tr}(\Gamma_{q} \otimes \ch^{*}_{K \to E}(\cE_{z})),
  \end{equation}
  where  $\mathrm{tr}$ denotes the integration over the maximal compact subgroup of $G$, the K-theoretic Gamma factor is
  \begin{equation}
    \Gamma_{q} = \frac{1}{\prod_{i=1}^{N}\phi(a_{i}^{-1}s^{-D_{i}}q^{q_{i}/2})},
  \end{equation}
  and $\cE_{z}$ is a unique solution to the interpolation problem
  \begin{equation}
    0 \to \cL \otimes \Theta(-V_{+}) \to \cL \otimes \cO_{Ell_{T \times G}(pt)} \to i^{*}\cL \otimes \cO_{Ell_{T}(\cX_{+})} \to 0
  \end{equation}
  for 
  \begin{equation}
  	\cE_+ \cdot \frac{\theta(z^{-1})}{\theta(s^{-1}z^{-1})} \in \Gamma(i^* \cL \otimes \cO_{Ell_T(\cX_+)}).
  \end{equation}

  The contour $\cC_{\delta}$ is a circle which is a shifted compact part of $G$. The shift is choosen to separate the poles from
  $\phi(V_{+}^{\vee})$ and $\phi(V_-^\vee)$.

  Moreover, if $Z^K(\cE_+, V_+^{\vee})$ does not converge, then the equality~\eqref{eq:solidTorus} should be
  understood as an asymptotic expansion as $z \to 0$.  	
\end{Theorem}
\begin{proof}
  We prove the theorem by computing the integral by residues and matching them with the contributions from
  the central charge.

  Let us deform the integration contour in the direction $s \to 0$.
  Then the contour hit the poles of
  \begin{equation}
    \frac{1}{\phi(V_{+}^{\vee})} = \frac{1}{\prod_{i \le N_{+}} \phi(a_{i}^{-1}s^{-D_{i}}q^{q_{i}/2})}.
  \end{equation}
  Notice that $\cE_z$ does not have any poles in $s$.
  The poles appear at
  \begin{equation}
    s = s_{k,\beta,m} = a_{k}^{1/D_{k}}q^{\beta} e^{\frac{2\pi \sqrt{-1} m}{D_k}}, \;\;\; k \le N_{+}, \; \beta \in \frac{1}{D_k}\bZ_{\ge 0}.
  \end{equation}
  
  Assume $\Gamma_{q} \cdot \cE_{z} \to 0, \; s \to 0$ on the complement to the neighbourhood of the poles. Then we have
  \begin{equation} \label{eq:coulomb1}
    \oint_{\cC_{\delta}} \frac{\dd s}{s} \; \Gamma_{q} \cdot \cE_{z} = \sum_{k \le N_{+}}\sum_{\beta \in \bZ_{\ge 0}/D_k}\sum_{0\le m \le D_k-1}
    \res_{s \to s_{k,\beta, m}} \frac{\dd s}{s} \, \Gamma_q \cdot \cE_z.
  \end{equation}
  In order to rewrite it as localization for the orbifold Euler characteristic we first shift 
  $s \to sq^{\beta}$:
  \begin{equation} \label{eq:coulomb2}
  	\oint_{\cC_{\delta}} \frac{\dd s}{s} \; \Gamma_{q} \cdot \cE_{z} = \sum_{k \le N_{+}}\sum_{\beta \in \bZ_{\ge 0}/D_k}\sum_{0\le m \le D_k-1}
  	\res_{s \to a_k^{1/D_k}e^{\frac{2\pi\sqrt{-1}m}{D_k}}} \frac{\dd s}{s} \, 
  	\frac{1}{\prod_{i \le N}\phi(a_i^{-1}s^{-D_i}q^{-d_i(\beta)})}\cdot \cE_z(q^{\beta}s).
  \end{equation}
  In order to compare with the Euler characteristic formula~\eqref{eq:orbEuler} we multiply and
  divide by the ideal of $K_T(\cX_{v(\beta)}) \; : \; \prod_{i, d_i(\beta) \in \bZ} (1-U_i)$,
  \begin{equation}
  	\frac{1}{\prod_{i \le N}\phi(a_i^{-1}s^{-D_i}q^{-d_i(\beta)})} = 
  	\frac{\prod_{i, \; d_i(\beta) \in \bZ} (1-U_i)}{\prod_i \phi(U_i q^{-d_i(\beta)})} 
  	\frac{1}{\prod_{i, \; d_i(\beta) \in \bZ} (1-U_i)}.
  \end{equation}

  Then we transform the elliptic brane factor. Since
  \begin{eqnarray}
    \cE_{z} \in \Gamma(\cU \otimes \Theta(V_{+}) ),
  \end{eqnarray}
  the $q^n$-shift of the elliptic brane is
  \begin{equation}
    \cE_{z}(s\zeta q^{n}) = z^{n}\prod_{i \le N_{+}}\frac{\theta(a_{i}^{-1}s^{-D_{i}}q^{-d_{i}(\beta)})}{\theta(a_{i}^{-1}s^{-D_{i}}q^{\{-d_{i}(\beta)\}})} \cdot\cE_{z}(s \zeta) = z^{n}\prod_{i \le N_{+}}\frac{\theta(a_{i}^{-1}s^{-D_{i}}q^{-d_{i}(\beta)})}{\theta(a_{i}^{-1}s^{-D_{i}}q^{\{-d_{i}(\beta)\}})} \cdot \frac{\theta(z^{-1})}{\theta(\zeta^{-1} s^{-1}z^{-1})} \cdot\cE_+.
  \end{equation}
  We further use the identity $\{-x\} = 1-\{x\}, \; x \notin \bZ$ to rewrite the last factor:
  \begin{equation}
  	\prod_{i \le N_{+}}\frac{\theta(a_{i}^{-1}s^{-D_{i}}q^{-d_{i}(\beta)})}{\theta(a_{i}^{-1}s^{-D_{i}}q^{\{-d_{i}(\beta)\}})} = \prod_{i, \; d_i(\beta) \in \bZ}(-U_i)^{-1}\prod_{i \le N_{+}}\frac{\theta(a_{i}^{-1}s^{-D_{i}}q^{-d_{i}(\beta)})}{\theta(a_{i}^{-1}s^{-D_{i}}q^{1-\{d_{i}(\beta)\}})}
  \end{equation}
  Collecting the results we obtain:
  \begin{equation}
  	\frac{1}{2\pi i} \oint \frac{\dd s}{s} \, \Gamma_q \cdot \cE_z = 
  	\sum_{k \le N_+}\sum_{\beta}\sum_m z^{\lfloor \beta \rfloor}\res_{s \to a_k^{1/D_k}e^{\frac{2\pi\sqrt{-1}m}{D_k}}}
  	\frac{\prod_{i, \, d_i(\beta) \in \bZ}(1-U_i^{-1})}{\prod_{i \le N}\phi(a_i^{-1}s^{-D_i}q^{-d_i(\beta)})}\frac{\theta(z^{-1})}{\theta(q^{-\{\beta\}}s^{-1}z^{-1})}\frac{\cE_+}{\prod_{i, \, d_i(\beta) \in \bZ}(1-U_i^{-1})}
  \end{equation}	

  Using the formulas~\eqref{eq:orbEuler},~\eqref{eq:hFunction} we can represent the integral~\eqref{eq:coulomb1} as
  \begin{equation}
    \sum_{\beta \in \mathrm{Eff}}z^{\lfloor\beta\rfloor}\frac{\theta(z^{-1})}{\theta(q^{-\{\beta\}}s^{-1}z^{-1})}\chi\left((\cH^{K}_{V_{+}})_{\beta} \otimes \ch^{E \to K}_{\mathrm{orb}}(\cE_{+})\right) =
    \chi(\cH^{K}_{V_{+}} \otimes \ch^{E\to K}_{\mathrm{orb}}(\cE_{+})).
  \end{equation}
\end{proof}

\begin{remark}
  We do not have to require $\cL$ to be unique solution of the interpolation problem. More generally, let $\cL \in \cU \otimes \Theta(R)$ for $R
  \in K_{T}([V/\Gamma])$. Define
  \begin{equation}
    \cE_{+} = \cL \cdot \frac{\theta(s^{-1}z^{-1})}{\theta(z^{-1})}\bigg|_{Ell_{T}(\cX_{+})} \in \Gamma(\cO_{Ell_{T}(\cX_{+}) \times E_{G}^{\vee}})_{\mathrm{mero}},
  \end{equation}
  where $\mathrm{mero}$ means that the section might have poles in $z$. By definition $\cE_{z}$ is some solution to the interpolation problem
  for $\cE_{+} \theta(z^{-1})/\theta(s^{-1}z^{-1})$.
  Then assuming convergence we have
  \begin{equation}
    Z^{K}(\cE_{+}, R) = \frac1{2\pi i} \oint_{\cC_{\delta}} \frac{\dd s}{s} \Gamma_{q} \cdot \cE_{z}.
  \end{equation}
  The only place where the proof above changes is the transformation factor computation for $\cE_{z}$ which is consistent with the definition of
  level $R$-structure.
\end{remark}

Last theorem provides a way to relate objects in the phases $\cX_{\pm}$ with objects on $[V/G]$.
By invoking this relation twice we can obtain the wall crossing statement: 

\begin{definition}
	Let $\cL = \cU \otimes \Theta(V_+)$, $\cE_+ \in \Gamma(\cO_{Ell_T(\cX_+)})$ and $\cE_z \in \Gamma(\cL)$ is
	 the unique interpolation of 
	 $$\cE_+ \cdot	\frac{\theta(z^{-1})}{\theta(s^{-1}z^{-1})}\bigg|_{\cX_+}.$$
	 Then 
	 $$\cE_- := E_z|_{\cX_-} \cdot \frac{\theta(s^{-1}z^{-1})}{\theta(z^{-1})}\bigg|_{\cX_-} \in 
	 \Gamma(\cO_{Ell_{T}(\cX_{-}) \times E_G^\vee})$$
	 is called the wall crossing brane of $\cE_+$. Note that it is a meromorphic section in $z$ as opposed
	 to $\cE_+$ that does not depend on $z$ at all.
\end{definition}
Note that multiplication by $\theta(z^{-1})/\theta(s^{-1}z^{-1})$ and interpolation above do not commute.\\

The following proposition is a consequence of definition above and theorem~\ref{th:solidTorus} applied to both $\cX_+$ and $\cX_-$.
\begin{Proposition}[Elliptic Wall Crossing] \label{th:wallCrossing}
	Let $\cE_+ \in \Gamma(\cO_{Ell_T(\cX_+)})$ and $\cE_- \in \Gamma(\cO_{Ell_{T}(\cX_{-}) \times E_G^\vee})$ be its wall crossing brane.
	 Then  $Z^K(E_-, V_+^{\vee})$ is the analytic continuation of $Z^K(E_+, V_+^{\vee})$ to the region $z \to \infty$ if it converges. Otherwise it is an asymptotic expansion
	 in $z \to \infty$.	
\end{Proposition}

\subsection{Quantum difference equation}

K-theoretic $I$-function satisfies a certain difference equation that can be thought of as a generalization of Picard-Fuchs differential equation
or Dubrovin (quantum) differential equation. This is evident in the Coulomb branch representation. Without loss of generality let $\cX = \cX_{+}$.

Let $T_{s} \; : \; s \to qs$ and $T_{z} \; : \; z \to qz$ be q-difference operators acting on $s$ and $z$ respectively.
By the previous section we know that
\begin{equation}
  Z^{K}(\cE_{+}, V_{+}^{\vee}) = \sum_{k \le N_{+}}\sum_{\beta \in \bZ/D_k}\sum_{\zeta,\zeta^{D_k}=1} res_{s \to \zeta a_{k}^{-1} q^{\beta}}
  \frac{\dd s}{s} \, \Gamma_{q} \cdot \cE_{z},
\end{equation}
where we expanded summation in $n$ to all integers (integrand does not have poles at $n \in \bZ_{<0}$).
By changing integration variable $s \to qs$ we can write
\begin{equation} \label{eq:residueShift}
  Z^{K}(\cE_{+}, V_{+}^{\vee}) = \sum_{k \le N_{+}}\sum_{\beta \in \bZ/D_k}\sum_{\zeta,\zeta^{D_k}=1} res_{s \to a_{k}^{-1}\zeta q^{\beta}} res_{s \to \zeta a_k^{-1} q^{\beta}}
  \frac{\dd s}{s} \, T_{s}(\Gamma_{q} \cdot \cE_{z}).
\end{equation}
Let us compute the action of $T_{s}$ on the integrand. We use
\begin{equation}
	\begin{aligned}
	   	&T_{x} \phi(x) = \frac{\phi(x)}{1-x}, \;\;\; T_{x}\theta(x) = (-x)^{-1}\theta(x).\\
	   	&\phi(q^n U) = \frac{\phi(U)}{(U;q)_n}, \\
	   	&\theta(q^n U) = U^{-n}q^{-\frac{n^2}{2}}(-q^{-\frac12})^{-n} \cdot \theta(U), \;\;\; n > 0, \\
	   	&\phi(q^{-n} U) = (q^{-n}U;q)_n \cdot \phi(U), \\
	   	&\theta(q^{-n} U) = U^{n}q^{-\frac{n^2}{2}}(-q^{-\frac12})^{n} \cdot \theta(U), \;\;\; n > 0.
 	\end{aligned}
\end{equation}
Let $V_{i} := U_{i}q^{q_{i}/2}$.
Then
\begin{equation}
  T_{s} \Gamma_{q} = T_{s}\frac{1}{\phi(V^{\vee})} = \frac{1}{\prod_{i}\phi(q^{-D_{i}}V_{i})} = \frac{\prod_{i > N_+}(V_i;q)_{-D_i}}{\prod_{i \le N_+}(q^{-D_i}V_{i};q)_{D_{i}}}  \cdot \Gamma_{q}.
\end{equation}
Moreover,
\begin{equation}
  \cE_{z} \in \Gamma(\cU \otimes \Theta(V_{+})).
\end{equation}
Operator $T_{s}$ is scalar on sections of both $\cU, \; \Theta(V_{+})$. We have
\begin{equation} \label{eq:diffPoincare}
  T_{s} \; \frac{\theta(sz^{-1})}{\theta(s)\theta(z^{-1})} = z, \;\;\; T_{z} \; \frac{\theta(sz^{-1})}{\theta(s)\theta(z^{-1})} = s
\end{equation}
\begin{equation}
  T_{s} \; \prod_{i \le N_{+}} \theta(V_{i}) =
  \prod_{i \le N_{+}} \theta(q^{-D_{i}}V_{i}) = \prod_{i \le N_{+}}V_{i}^{D_{i}}q^{-\frac{D_i^2}2}(-q^{-\frac12})^{D_i} \cdot \prod_{i \le N_{+}}\theta(V_{i}).
\end{equation}
Using this we find
\begin{equation} \label{eq:Tintegrand}
  T_{s} \, (\Gamma_{q} \cdot \cE_{z}) = z \prod_{i \le N_{+}}\frac{q^{-\frac{D_i(D_i+1)}{2}}(-V_{i})^{D_{i}}}{\prod_{k=0}^{D_i-1} (1-q^{k-D_i}V_i)}
  \prod_{i > N_+}(V_i;q)_{-D_i} \cdot (\Gamma_{q} \cdot \cE_{z}),
\end{equation}
where
\begin{equation}
  \frac{q^{-\frac{D_i(D_i+1)}{2}}(-V_i)^{D_i}}{\prod_{k=1}^{D_i}(1-q^{-k}V_i)} =
  \frac{q^{-\frac{D_i(D_i+1)}{2}}(-V_i)^{D_i}}{\prod_{k=1}^{D_i}q^{-k}(-V_i)(1-q^{k}V_i^{-1})} 
  = \frac{1}{\prod_{k=0}^{D_i-1} (1-q^{1+k}V_i^{-1})}.
\end{equation}
Thus, we can simplify~\eqref{eq:Tintegrand}:
\begin{equation}
	T_{s} \, (\Gamma_{q} \cdot \cE_{z}) = z \frac{\prod_{i > N_+}(V_i;q)_{-D_i}}{\prod_{i \le N_{+}} (qV_i^{-1};q)_{D_i}}
	 \cdot (\Gamma_{q} \cdot \cE_{z}).
\end{equation}
Using the the second part of~\eqref{eq:diffPoincare} we can produce $U_i$ and $V_{i}$
\begin{equation}
	a_i^{-1}(T_z)^{-D_i} \cdot \cE_z = U_i\,  \cE_z, \;\;\; a_i^{-1}q^{q_{i}/2}(T_z)^{-D_i} \cdot \cE_z = V_i\,  \cE_z.
\end{equation}
Therefore, we can write $T_s (\Gamma_q \cdot \cE_z)$ in terms of $T_z$ as well:
\begin{equation}
  \prod_{i \le N_+} \prod_{k=1}^{D_i}(1-q^{k-q_{i}/2} a_i T_z^{D_i}) \,z^{-1} T_s \, (\Gamma_{q} \cdot \cE_{z}) =  \prod_{i > N_{+}}\prod_{k=0}^{-D_i-1}(1-q^{k+q_{i}/2} a^{-1}_{i}T_{z}^{-D_{i}})
  \cdot (\Gamma_q \cdot \cE_z),
\end{equation}
or commuting the difference operator in the left hand side with $z^{-1}$:
\begin{equation}
	z^{-1}\prod_{i \le N_+} \prod_{k=1}^{D_i}(1-q^{k-q_{i}/2-D_i} a_i T_z^{D_i}) \, T_s \, (\Gamma_{q} \cdot \cE_{z}) =  \prod_{i > N_{+}}\prod_{k=0}^{-D_i-1}(1-q^{k+q_{i}/2} a^{-1}_{i}T_{z}^{-D_{i}})
	\cdot (\Gamma_q \cdot \cE_z),
\end{equation}

We can use this relation and formula~\eqref{eq:residueShift} to get a q-difference equation
on $Z^K(\cE_+, V_+)$. 
so
\begin{equation}
  \left[\prod_{i \le N_+} \prod_{k=1}^{D_i}(1-q^{k-q_{i}/2-D_i} a_i T_z^{D_i}) -  z \prod_{i > N_{+}}\prod_{k=0}^{-D_i-1}(1-q^{k+q_{i}/2} a^{-1}_{i}T_{z}^{-D_{i}})\right]Z^{K}(\cE_{+}, V_{+}) = 0.
\end{equation}
This equation is called the quantum difference equation and is analogous to Picard-Fuchs differential
equation in the cohomological case.

\paragraph{\bf Order of QqDE and elliptic cohomology}

The order of this equation is $\mathrm{max}(\sum_{i \le N_+} D_i^2, \sum_{i > N_+} D_i^2)$. This is in
a contrast with the cohomological case (quantum differential equation) where the order is 
$\mathrm{max}(\sum_{i \le N_+} D_i, \sum_{i > N_+} (-D_i))$. Geometrically the reason is clear.
In cohomology each fixed point $[pt/\mu_{r}]$ contributes $\dim(H^*(\cI [pt/\mu_r])) = \dim(K([pt/\mu_r])) = r$ solutions to the
equation meanwhile in K-theory the same fixed point contributes 
$\dim(K(\cI [pt/\mu_r])) = \dim(\cO_{Ell([pt/\mu_r])}) = r^2$ solutions.

\section{Example: degree $r$ hypersurface in $\bP^{n-1}$}

Consider the corresponding GLSM data: $D_1, \ldots, D_{N_+} = 1, \; D_N = -r$. We have $\cX_+ \simeq \cO_{\bP^{N_+-1}}(-r)$ and
$\cX_- \simeq [\bC^{N_+}/\mu_r]$. The R-charge weights can be set as $(0,\cdots,0,1)$ and an example of 
the superpotential is $p(x_1^r+ \cdots + x_{N_+}^r)$, where $(x_1, \ldots, x_{N_+}, p)$ are coordinates on 
$V \simeq \bC^{N}$.

In particular, 
\begin{equation}
	Ell_T(\cX_+) = \{\prod_{i=1}^{N_+}\theta(a_i s) = 0\}	 \subset E_T \times E_G.
\end{equation}
This is a degree $N_+$ cover over $E_T$. Whereas
\begin{equation}
	Ell_T(\cX_-) = \{\theta(a_{N} s^r) = 0\}	 \subset E_T \times E_G.
\end{equation}
This is a degree $r^2$ cover over $E_T$. We see that dimensions of elliptic cohomology of $\cX_+$ and $\cX_-$ (over $Ell_T(pt)$) coincide only if $N_+ =r^2$. This is in contrast with the K-theory where the
crepant condition is $N_+=r$.\\

\paragraph{\bf Geometric phase}

Consider the following elliptic branes on $[V/G]$:
\begin{equation}
	\cE_z^{(+,k)} := \frac{\theta(qa_k s z^{-1})}{\theta(qa_k z^{-1})}\prod_{i \le N_+, i \ne k} \theta(q a_is) \in \Gamma(\cU \otimes \Theta(V_+)).
\end{equation}
Restricting to the fixed points $pt_j = \{s = a_j^{-1}\}$ for generic $z$ we get
\begin{equation}
	\begin{aligned}
		&\cE_{z}^{(+,k)}|_{pt_j} = 0, \;\;\; j \ne k \\
		&\cE_{z}^{(+,k)}|_{pt_k} = \frac{\theta(qz^{-1})}{\theta(qa_k z^{-1})}\prod_{i \ne k} \theta(a_ka_i^{-1}).
	\end{aligned}
\end{equation}
Thus, these branes are supported on the fixed points and provide a basis (over $\cO_{E_T}$) of elliptic
branes. The solid torus partition function is
\begin{multline} \label{eq:geometricZ}
	Z^K(\cE_z^{(+,k)}, V_+^{\vee}) = \frac{1}{2\pi i}\oint_{\cC_{\delta}} \frac{\dd s}{s} \, \frac{\prod_{i \ne k}\phi(qa_is)}
	{\phi(a_N^{-1}s^{r}q)} \cdot \frac{\theta(qa_ksz^{-1})}{\theta(qa_k z^{-1})\phi(a_k^{-1}s^{-1})} = \\
	 = \sum_{n \ge 0} \res_{s \to a_k^{-1}q^n}\left[\frac{\dd s}{s \phi(a_k^{-1}s^{-1})} \right]
	\frac{\prod_{i \ne k}\phi(q^{n+1}a_ia_k^{-1})}
	{\phi(q^{rn+1}a_N^{-1}a_k^{-r})} \frac{\theta(q^{n+1}z)}{\theta(qa_k z)}.
\end{multline}
Now we compute
\begin{equation}
	\frac{\theta(q^{n+1}z)}{\theta(qa_k z)} = (-z)^{-n}q^{-\frac{n(n+1)}2}\frac{\theta(z^{-1})}{\theta(a_k^{-1} z^{-1})}
\end{equation}
and
\begin{equation}
	\res_{s \to a_k^{-1}q^n}\left[\frac{\dd s}{s \phi(a_k^{-1}s^{-1})} \right] = 
	-\res_{s \to 1} \left[ \frac{\dd s}{s \phi(q^{-n}s)} \right] = -\frac{1}{\phi(q)(q^{-n};q)_n}\res_{s \to 1}\left[ \frac{\dd s}{s(1-s)} \right] = (-1)^n\phi(q)^{-1}\frac{q^{\frac{n(n+1)}{2}}}{(q;q)_n}.
\end{equation}
Thus, the central charge is
\begin{equation}
	Z^K(\cE_z^{(+,k)}, V_+^{\vee}) = \phi(q)^{-1}\sum_{n \ge 0} \frac{z^n}{(q;q)_n}
	\frac{\prod_{i \ne k}\phi(q^{n+1}a_ia_k^{-1})}
	{\phi(q^{rn+1}a_N^{-1}a_k^{-r})}  \frac{\theta(z^{-1})}{\theta(a_k^{-1} z^{-1})}.
\end{equation}
It satisfies the QqDE in the geometric phase:
\begin{equation}
	\left[\prod_{i \le n} (1-a_i T_z) - z \prod_{k=0}^{n-1}(1-q^{k+1} a^{-1}_{N}T_{z}^{r})\right]Z^{K}(\cE_{+}, V_{+}) = 0,
\end{equation}
which can be also checked by a direct computation.

\paragraph{\bf Landau-Ginzburg phase}
In the other phase we can choose the following set of branes:
\begin{equation}
	\cE_z^{(-,\zeta)} := \frac{\theta(\zeta^{-1} a_N^{1/r} s^{-1} z)}{\theta(\zeta^{-1} a_N^{1/r} z)\theta(\zeta^{-1} a_N^{1/r} s^{-1})}\theta(a_Ns^{-r}) \in \Gamma(\tilde{\cU} \otimes \Theta(V_-)),
\end{equation}
where $[\zeta] \in E[r]$.
Restricting to the twisted sectors $s = \xi a_N^{1/r}$ of the unique fixed point we obtain
\begin{equation}
	\begin{aligned}
		&\cE_{z}^{(-,\zeta)}|_{pt_{\xi}} = 0, \;\;\; \zeta \ne \xi \\
		&\cE_{z}^{(-,\zeta)}|_{pt_\zeta} = \frac{\theta(z)}{\theta(\zeta^{-1} a_N^{1/r} z)}\lim_{s \to 1}\frac{\theta(\zeta^{r}s^{-r})}{\theta(s)} = -r\zeta^{r}\frac{\theta(z)}{\theta(\zeta^{-1} a_N^{1/r} z)}\frac{\theta'(\zeta^{r})}{\theta'(1)}.
	\end{aligned}
\end{equation}
We can also simplify the expression above using
\begin{equation}
	\theta'(q^{-n}) = -(q^{-1};q^{-1})_n\phi^2(q) = (-1)^{n+1}q^{-\frac{n(n+1)}{2}}(q;q)_n \phi^2(q).
\end{equation}

The solid torus partition function is
\begin{multline}
	Z^K(\cE_z^{(-,\zeta)}, V_-^{\vee}) = \frac{1}{2\pi i}\oint_{\cC_{\delta}} \frac{\dd s}{s} \, \frac{\phi(a_Ns^{-r})}
	{\prod_{i \le N_+}\phi(a_i^{-1}s^{-1})} \cdot \frac{\theta(\zeta^{-1} a_N^{1/r} s^{-1} z)}{\theta(\zeta^{-1} a_N^{1/r} z)\theta(\zeta^{-1} a_N^{1/r} s^{-1})} = \\
	= \sum_{n >  0} \res_{s \to \zeta^{-1} a_N^{1/r}q^{-n}}\left[\frac{\dd s}{s \, \theta(\zeta^{-1} a_N^{-1/r}s^{-1})} \right]
	\frac{\phi(\zeta^{r} q^{rn})}
	{\prod_{i \le N_+}\phi(\zeta a_i^{-1}a_N^{1/r}q^{n})} \frac{\theta(q^{n}z)}{\theta(\zeta^{-1} a_N^{1/r} z)}.
\end{multline}

We compute:
\begin{equation}
	\res_{s \to \zeta^{-1} a_N^{1/r}q^{-n}}\left[\frac{\dd s}{s \, \theta(\zeta^{-1} a_N^{1/r}s^{-1})} \right] =
	\res_{s \to 1}\left[ \frac{\dd s}{s \, \theta(q^{n}s)} \right] = q^{\frac{n^{2}}2}(-q^{-\frac12})^{n}\res_{s \to 1}\frac{\dd s}{s \, \theta(s)} =
  -q^{\frac{n^{2}}{2}}(-q^{-\frac12})^{n}\phi(q)^{-2}.
\end{equation}
Then
\begin{equation}
  Z^K(\cE_z^{(-,\zeta)}, V_-^{\vee}) = -\phi(q)^{-2}\sum_{n > 0} z^{-n}
	\frac{\phi(\zeta^{r} q^{rn})}
	{\prod_{i \le N_+}\phi(\zeta a_i^{-1}a_N^{-1/r}q^{n})} \frac{\theta(z)}{\theta(\zeta^{-1} a_N^{1/r} z)}.
\end{equation}

It satisfies the QqDE:
\begin{equation}
  \left[\prod_{k=1}^{r}(1-q^{k-r} a_N T_z^{-r}) -z \prod_{i \le N_+}(1- a^{-1}_{i}T_{z}^{-1})\right]Z^{K}(\cE_{-}, V_{-}) = 0.
\end{equation}
As we see, these equations are of order $\max(r^{2}, N_{+})$ as opposed to order $\max(r,n)$ of the usual Picard-Fuchs
equation. For the quintic threefold this phenomenon was noticed, for example in~\cite{Wen}.
The reason being that $\dim(K(\cI [\bC^{N_{+}}/\mu_r])) = r^{2}$ as opposed to $\dim(K([\bC^{N_{+}}/\mu_r]))=r$.

{\bf Wall Crossing}
We can also deform the contour in~\eqref{eq:geometricZ} to $s \to \infty$.
Wall-crossing of the branes $\cE^{(+,k)}_{z}$ has the form
\begin{equation}
	\frac{\theta(s^{-1}z^{-1})}{\theta(z^{-1})}\cE_z^{(+,k)} = \frac{\theta(s^{-1}z^{-1})\theta(qa_k s z^{-1})}{\theta(z^{-1})\theta(qa_k z^{-1})}\prod_{i \le N_+, i \ne k} \theta(q a_is) \in \Gamma(\cO_{Ell_{\cX_{-}} \times E^{\vee}_{G}}),
\end{equation}
and wall-crossing of $\cE^{(-,k)}_{z}$ are
\begin{equation}
  \frac{\theta(s^{-1}z^{-1})}{\theta(z^{-1})}\cE_z^{(+,k)} = \frac{\theta(qsz)\theta(\zeta^{-1} a_N^{1/r} s^{-1} z)}{\theta(qz)\theta(\zeta^{-1} a_N^{1/r} z)}\frac{\theta(a_Ns^{-r})}{\theta(\zeta^{-1} a_N^{1/r} s^{-1})} \in \Gamma(\cO_{Ell_{\cX_{+}} \times E^{\vee}_{G}}).
\end{equation}

\appendix

\section{Elliptic cohomology}
There are various constructions of elliptic cohomology that serve different purposes.
We follow expositions of~\cite{GKV} and~\cite{AO, Inductive, Nonabelian} and we refer
the reader there for more detailed exposition of the subject. Our notatins are closer to
\cite{AO, Inductive, Nonabelian} rather than \cite{GKV}. Here we present a very brief introduction
to the subject.\\

Given a Lie group $G$ (which is abelian in this paper), equivariant cohomology and K-theory of a
smooth DM stack are rings of $H_{G}(pt)$ (respectively $K_{G}(pt)$) modules. Equivariant cohomology
is not a ring of modules but rather a (sheaf on a) projective scheme, so it is more convenient to use the
dual language of spectra. For example, spectrum of equivariant cohomology (K-theory) is an affine scheme
over $\mathrm{Lie}(G) \simeq \Spec(H_{G}(pt))$ (respectively $G \simeq \Spec(K_{G}(pt))$).\\

Let $E \simeq \bC^{*}/q^{\bZ}$ be an elliptic curve, where we can think of $q$ as either a complex number
or a formal parameter. In the study of elliptic cohomology the modularity properties (that correspond to
the $q$-dependence) play an important role, but we do not study it here. We define G-equivariant elliptic
cohomology of a point to be a projective scheme:
\begin{equation}
  Ell_{G}(pt) = E_{G} := Bun_{G}(E^{\vee}).
\end{equation}
Two main examples for this paper are $G = T \simeq (\bC^{*})^{n}$ and $G = \mu_{n}$:
\begin{equation}
  E_{T} \simeq T/q^{\mathrm{cochar}(T)}, \;\;\; E_{\mu_{n}} = E[n],
\end{equation}
where $E[n]$ is a group of $n$-torsion points on $n$, or simply $n$-th roots of unity on $E$.
For a general smooth DM stack $X$
elliptic cohomology is a projective scheme $$Ell_{G}(X) \stackrel{\pi_{X}}{\to} Ell_{G}(pt).$$
One should think of the scheme above as generalization
of $\Spec(K_{G}(X)) \to \Spec(K_{G}(pt))$.

Elements of elliptic cohomology are
sections of line bundles over $Ell_{G}(X)$. In particular, the pushforward of
structure sheaf $(\pi_{X})_{*}\cO_{Ell_G(X)}$ is a sheaf of $\cO_{E_{G}}$-modules that satisfies a set
of axioms for equivariant generalized cohomology theories~\cite{GKV}.

One can define $Ell_{G}(X)$ either axiomatically or provide explicit construction~\cite{Grojnowsky}.
Another useful fact is that elliptic cohomology fits into the natural diagram:
\begin{equation}
    \begin{tikzcd}
      \Spec(H^{*}_{T}(\cX)) \arrow[d] \arrow[r, "\ch_{H \to K}"] &\Spec(K_{T}(X)) \arrow[d] \arrow[r, "\ch_{K \to E}"]
      & Ell_{T}(X) \arrow[d, "\pi_{X}"]\\
      \mathrm{Lie}(T) \arrow[r, "\exp"]& T \arrow[r, "\mod \; q^{\mathrm{cochar}(T)}"]& E_{T}
    \end{tikzcd}
  \end{equation}


\end{document}